\documentclass[11pt]{article}
\usepackage{amsmath,amssymb,amsthm,color,graphicx}
\usepackage{accents}
\usepackage[unicode,breaklinks=true,colorlinks=true]{hyperref}

\usepackage[top=1in, bottom=1in, left=1.25in, right=1.25in]{geometry}

\newcommand{\snb}[1]{}

\numberwithin{equation}{section}
\newtheorem{theorem}{Theorem}[section]
\newtheorem*{theorem*}{Theorem}
\newtheorem{corollary}[theorem]{Corollary}
\newtheorem{lemma}[theorem]{Lemma}
\newtheorem{proposition}[theorem]{Proposition}
\newtheorem{definition}[theorem]{Definition}

\theoremstyle{remark}

\newtheorem*{remark*}{Remark}

\newcommand{\bke}[1]{\left ( #1 \right )}
\newcommand{\bkt}[1]{\left [ #1 \right ]}
\newcommand{\bket}[1]{\left \{ #1 \right \}}
\newcommand{\norm}[1]{ \| #1  \|}

\newcommand{\abs}[1]{\left | #1 \right |}

\newcommand\al{\alpha}
\newcommand\be{\beta}
\newcommand\ga{\gamma}
\newcommand\de{\delta}

\newcommand\ve{\varepsilon}
\newcommand\e {\epsilon}

\renewcommand\th{\theta}

\newcommand\la{\lambda}
\newcommand\si{\sigma}

\newcommand\om{\omega}

\newcommand\De{\Delta}

\newcommand\Om{\Omega}

\newcommand{\R}{\mathbb{R}}
\newcommand{\RR}{\mathbb{R}}

\newcommand{\ZZ}{\mathbb{Z}}
\newcommand{\NN}{\mathbb{N}}

\renewcommand{\div}{\mathop{\rm div}}
\newcommand{\curl} {\mathop{\rm curl}}
\newcommand{\supp} {\mathop{\mathrm{supp}}}
\newcommand{\spt}{\mathop{\mathrm{spt}}}

\newcommand{\esssup} {\mathop{\rm ess\,sup}}

\newcommand{\pd}{\partial}
\newcommand{\nb}{\nabla}
\newcommand{\td}{\tilde}

\newcommand{\lec}{{\ \lesssim \ }}

\newcommand{\oo}{\infty}
\newcommand{\EQ}[1]{\begin{equation}\begin{split} #1 \end{split}\end{equation}}
\newcommand{\EQN}[1]{\begin{equation*}\begin{split} #1 \end{split}\end{equation*}}

\newcommand{\I}{\infty}
\newcommand{\loc}{\mathrm{loc}}
\newcommand{\uloc}{\mathrm{uloc}}
\newcommand{\far}{\mathrm{far}}
\newcommand*{\dt}[1] {\accentset{\mbox{\large\bfseries .}}{#1}}
\renewcommand{\dot}{\dt}
\newcommand{\wkto}{\rightharpoonup}
\newcommand{\longrightharpoonup}{\relbar\joinrel\rightharpoonup}

\begin{document}

\title{Short time regularity of Navier-Stokes flows \\ with locally $L^3$ initial data
and applications}
\author{Kyungkeun Kang%
\thanks{Department of Mathematics, Yonsei University, Seoul 120-749,
South Korea. Email: \texttt{kkang@yonsei.ac.kr}}
\and
Hideyuki Miura%
\thanks{Department of Mathematical and Computing Sciences, Tokyo Institute of Technology, Tokyo 152-8551, Japan. Email: \texttt{miura@is.titech.ac.jp}}
\and
Tai-Peng Tsai%
\thanks{Department of Mathematics, University of British Columbia,
Vancouver, BC V6T 1Z2, Canada. Email:  \texttt{ttsai@math.ubc.ca}
}}
\date{} %
\maketitle
\begin{abstract}
We prove
short time regularity of
suitable weak solutions
of 3D incompressible
Navier-Stokes equations near a point where the initial data is locally in $L^3$.
The result is applied to the regularity problems of solutions with uniformly small local $L^3$ norms, and of forward discretely self-similar solutions.


\end{abstract}

\section{Introduction}

The Navier-Stokes equations describe the evolution of a viscous incompressible fluid's velocity field $v$ and its associated scalar pressure $\pi$.  They are required to satisfy
\begin{equation}\label{NS}\tag{\textsc{ns}}
\partial_tv-\Delta v +v\cdot\nabla v+\nabla \pi = 0,
\quad
\nabla \cdot v =0,
\end{equation}
in the sense of distributions.  For our purposes, \eqref{NS} is applied on $\R^3\times (0,\I)$ and $v$ evolves from a prescribed, divergence free initial data $v_0:\R^3\to \R^3$.  Solutions to \eqref{NS} satisfy a
natural scaling: if $v$ satisfies \eqref{NS}, then for any $\lambda>0$
\begin{equation}
    v^{\lambda}(x,t)=\lambda v(\lambda x,\lambda^2t),
\end{equation}
is also a solution with pressure
\begin{equation}
    \pi^{\lambda}(x,t)=\lambda^2 \pi(\lambda x,\lambda^2t),
\end{equation}
and initial data
\begin{equation} \label{v0scaling}
v_0^{\lambda}(x)=\lambda v_0(\lambda x).
\end{equation}
A solution is called self-similar (SS) if $v^\lambda(x,t)=v(x,t)$ for all $\lambda>0$ and is discretely self-similar with factor $\lambda$ (i.e.~$v$ is $\lambda$-DSS) if this scaling invariance holds for a given $\lambda>1$. Similarly, $v_0$ is self-similar (a.k.a.~$(-1)$-homogeneous) if $v_0(x)=\lambda v_0(\lambda x)$ for all $\lambda>0$ or $\lambda$-DSS if this holds for a given $\lambda>1$.
These solutions can be either forward or backward if they are defined on $\R^3\times (0,\I)$ or $\R^3\times (-\I,0)$ respectively.  In this paper we work exclusively with forward solutions and omit the qualifier ``forward''.

Self-similar solutions are interesting in a variety of contexts as candidates for ill-posedness  or finite time blow-up of solutions to the 3D Navier-Stokes equations (see \cite{GuiSve,JiaSverak,JiaSverak2,leray,NRS,Tsai-ARMA} and the discussion in \cite{BT1}. Forward self-similar solutions are compelling candidates for non-uniqueness \cite{JiaSverak2,GuiSve}. Until recently, the existence of forward self-similar solutions was only known for small data (see the references in \cite{BT1}). Such solutions are necessarily unique.  In \cite{JiaSverak},  Jia and \v Sver\'ak constructed forward self-similar solutions for large data where the data is assumed to be H\"older continuous away from the origin.
This result has been generalized in a number of directions by a variety of authors \cite{BT1,BT2,BT3,BT5, Chae-Wolf,KT-SSHS,LR2,Tsai-DSSI}; see also the survey \cite{JST}.

The motivating problem for this paper is the following question: It is shown in Tsai \cite{Tsai-DSSI} that, if a
$\la$-DSS initial data $v_0\in C^\al_{\loc}(\R^3 \setminus \{0\})$, $0<\al<1$, with $M=\norm{v_0}_{ C^\al(B_2 \setminus B_1)}<\infty$, and if $\la-1\le c_1(M)$ for some sufficiently small positive constant $c_1$ depending on $M$, then there is a $\la$-DSS solution $v$ with initial data $v_0$ such that $v$ is regular, that is, $v \in L^\infty_\loc(\R^4_+)$.
The question is: What if we weaken the assumption of $v_0$ so that $v_0$ is in $L^{3,\infty}(\R^3)$ (weak $L^3$)?
Note that for $v_0 \in L^{3,\infty}(\R^3)$ that is $\la$-DSS and divergence free,  Bradshaw and Tsai \cite{BT1} constructs at least one $\la$-DSS local Leray solution, based on a weak solution approach. Thus regularity cannot be obtained as a by-product of the existence proof, and one needs to prove the regularity of any such solutions if $\la-1$ is sufficiently small.

\medskip

Motivated by this problem, we need to study solutions whose initial data is locally in $L^3$, as it is also shown in \cite{BT1} that, when $v_0$ is $\la$-DSS, then $v_0 \in L^{3,\infty}(\R^3)$ if and only if $v_0 \in L^3(B_\la \setminus B_1)$.
The following local theorem is our first main result.

\begin{theorem} \label{th1.1}

There are positive constants $\epsilon_0$, $C_1$
such that the following holds.
For any $M>0$, there exist $T_1=T_1(M)\in (0,1]$ 
such that, if
$(v,\pi)$ is any suitable weak solution of the Navier-Stokes equations \eqref{NS} in $B_1 \times (0,T_1)$
with initial data $v_0$,
\begin{equation}\label{th1.1-0a}
\norm{v_0}_{L^3(B_1)}\le \epsilon_0,
\end{equation}
and
\begin{equation}\label{th1.1-0}
\norm{v}_{L^{\infty}_tL^2_x\cap L^2_tH^1_x (B_1 \times (0,T_1))}^2
+ \norm{\pi}_{L^{2}_t L^{3/2}_x(B_1 \times (0,T_1))}
\le M,
\end{equation}
then
$v$ is regular in $B_{1/4}  \times (0,T_1)$ with
\begin{equation}\label{th1.1-2}
|v(x,t)| \le \frac {C_1}{\sqrt t},\quad \text{in}\quad B_{1/4}  \times (0,T_1),
\end{equation}
\begin{equation}\label{th1.1-1}
\sup_{z_0  \in B_{\frac14}\times (0,T_1)} \sup_{0<r<\infty}\frac{1}{r^{2}}\int_{Q_{z_0,r} \cap [B_1 \times (0,T_1)]}\abs{v}^3
dz\le 1.
\end{equation}
We can choose $T_1(M)=\e (1+M)^{-6}$ for some sufficiently small $\e$ independent of $M$.  
\end{theorem}
Above, we define the parabolic cylinder by
$Q_{z_0,r} = B_r(x_0) \times (t_0-r^2,t_0)$ for 
$z_0 = (x_0,t_0)$.

\noindent
{\it Comments for Theorem \ref{th1.1}:}

\begin{enumerate}
\item Our result holds for locally (in space) defined suitable weak solutions. In particular, no boundary condition is assumed on 
$\pd B_1 \times (0,T_1)$.

\item The quantities bounded by $M$ in \eqref{th1.1-0} both have dimension 1 in the sense of \cite{CKN}. This is convenient for the tracking of constants in Corollary \ref{cor1.2}.

\item It should be noted that the constant $C_1$ is independent of $M$. Intuitively, the nonlinear term has no effect before $T_1=T_1(M)$, and hence the solution behaves like a linear solution, and its size is given by the initial data.

\item The boundedness of $\pi$ in $L^{2}_t L^{3/2}_x$ is natural for the Leray-Hopf weak solutions defined in $\R^3$, 
as $\pi$ is given by $\pi=R_iR_j(v_iv_j)$, where $R_j = (-\De)^{-1/2} \pd_j$ is the Riesz transform, and
\[
\norm{\pi }_{ L^{2}_t L^{3/2}_x(\R^3 \times (0,T))}\le C \norm{v }_{ L^{4}_t L^{3}_x(\R^3 \times (0,T))}^2
\le C \norm{v}_{L^\infty L^2 \cap L^2 \dot H^1(\R^3 \times (0,T))}^2.
\]
We will prove the local-in-space pressure bound for local energy solutions in Lemma \ref{apriori-bounds}.

\item
The assumption  $\norm{\pi }_{ L^{2}_t L^{3/2}_x(B_1 \times (0,T_1))}\le M$ can be replaced by, e.g.,~$\norm{\pi}_{L^{q}(Q)}
\le M$ for some $q\in(3/2,5/3]$.  It ensures that $\int_0^T \int_{B_1} |v|^3+|p|^{3/2}$ is small for sufficiently small $T=T(M)$; thus $q=3/2$ is not allowed.
Our choice of exponents is to maximize the time exponent, so that $T_1(M)=\e (1+M)^{-m}$ has the smallest $m=6$.

\item Theorem \ref{th1.1} is an extension of Jia-Sverak \cite[Theorem 3.1]{JiaSverak}, in which the initial data is assumed in $L^m(B_1)$, $m>3$. This is similar to the extension of the mild solution theory for the scale subcritical data $v_0\in L^m(\R^3)$, $m>3$, of
Fabes-Jones-Rivi{\`e}re \cite{FJR}  to the critical data
$v_0\in L^3(\R^3)$ of Weissler \cite{Weissler}, Giga-Miyakawa \cite{GM}, Kato \cite{Kato84} and Giga \cite{Giga}.
\end{enumerate}

Our first set of applications (Corollaries \ref{cor1.2}-\ref{cor1.4}) of Theorem \ref{th1.1} is concerned with   \emph{local energy solutions} defined globally in $\R^3$, which are weak solutions
of \eqref{NS} in $\R^3 \times (0,\infty)$ that are uniformly in time bounded in $L^2_\uloc$, and satisfying the local energy inequality. 
See Section \ref{S3} for their definitions and properties. 
In order to state our result, we introduce the uniformly local $L^q$ spaces.
For $q \in [1,\infty)$, we say $f\in L^q_{\uloc}$ if $f\in L^q_\loc(\R^3)$ and
\EQ{\label{Lq-uloc}
\norm{f}_{L^q_{\uloc}} =\sup_{x \in\R^3} \norm{f}_{L^q(B_1(x))}<\infty.
}
We also denote for $\rho>0$
\[
\norm{f}_{L^q_{\uloc,\rho}} =\sup_{x \in\R^3} \norm{f}_{L^q(B_{\rho}(x))}.
\]
Let $E^q$ be the closure of $C^\infty_c(\R^3)$ in $L^q_{\uloc}$-norm. Equivalently, $E^q$ consists of those $f \in L^q_{\uloc}$ with $\lim_{|x|\to \infty} \norm{f}_{L^q(B_1(x))} = 0$, see \cite{LR}.

\medskip

In the following corollary we assume that the initial data  belongs to $L^3(B_{\delta}) \cap E^2$.

\begin{corollary}\label{cor1.2}
Let $\epsilon_0$ and $C_1$ be the constants from Theorem \ref{th1.1}.
Suppose $v$ is a local Leray solution of the Navier-Stokes equations \eqref{NS}
with initial data $v_0\in E^2$ and there is $\de \in (0,1]$ such that
\begin{equation}
\norm{v_0}_{L^3(B_\de)}\le \epsilon_0.
\end{equation}
Then, there exist $T_2=T_2(\de,\norm{v_0}_{L^2_{\uloc}})>0$, such that
$v$ is regular in $B_{\de/4}  \times (0,T_2)$ with
\[
|v(x,t)| \le \frac {C_1}{\sqrt t}, \quad \text{in }\ B_{\de/4} \times (0,T_2),
\]
and
\[\sup_{z_0  \in B_{\de/4}\times (0,T_2)} \sup_{0<r<\infty}\frac{1}{r^{2}}\int_{Q_{z_0,r} \cap [B_{\de/4}  \times (0,T_2)]}\abs{v}^3
dz\le 1.
\]
Furthermore, we can take
$T_2=T_1(M)\de^2$ with
$M = \frac C\de \sup_{x_0}\int_{B_\de(x_0)}|v_0|^2$.
\end{corollary}

\medskip

In Corollary \ref{cor1.3} we assume the initial data  $v_0 \in L^3_{\uloc}(\R^3)\cap E^2$.

\begin{corollary}\label{cor1.3}
Let $\e_0$ be the small constant from Theorem \ref{th1.1}.
Suppose $v$ is a local Leray solution of the Navier-Stokes equations \eqref{NS}
with initial data $v_0\in L^3_\uloc\cap E^2$ and there is $\de \in (0,\infty)$ such that
\begin{equation}\label{cor1.3-1}
\sup_{  x_0 \in \R^3}\int_{B_{\delta}(x_0)}\abs{v_0}^3\le
\epsilon_0^3.
\end{equation}
Then, there is $T>0$ such that $v$ is regular in $\R^3 \times (0,T)$ with
\begin{equation}\label{cor1.3-2}
|v(x,t)| \le \frac C{\sqrt t}, \quad (0<t<T).
\end{equation}
\end{corollary}

This result is similar to Maekawa-Terasawa \cite[Theorem 1.1 (iii)]{MaTe}.
Indeed, \cite{MaTe} constructs local in time mild solutions in the intersection of $L^\infty(0,T; L^3_\uloc)$ and \eqref{cor1.3-2},
for $v_0\in L^3_\uloc$ that satisfies the smallness condition \eqref{cor1.3-1}.
They have $T= C\de^2 \norm{u}_{L^3_{\uloc,\de}}^{-4}$, and do not assume spatial decay $v_0\in E^2$.

In constrast, Corollary \ref{cor1.3} is a regularity theorem, assuming further the spatial decay of $v_0$.

In Corollary \ref{cor1.4} we consider general initial data  $v_0 \in E^2$. Let
\[
 \rho(x;v_0) =\sup \bket{ r>0: v_0 \in L^3(B_r(x)), \int_{B_r(x)} |v_0|^3 \le \e_0^3}.
\]
We let $\rho(x;v_0)=0$ if such $r$ does not exist.

\begin{corollary}
\label{cor1.4} Suppose $v_0\in E^2$ and $\div v_0=0$.  Let $\bar \rho(x)
= \min(\rho(x;v_0),1)\ge0$, and $N_r = \sup_{x_0\in \RR^3} \frac 1r \int_{B_r(x_0)}  {|v_0|^2}\,dx$. Let
\[
T(x) = \e (1+N_{\bar \rho(x)})^{-6} \bar \rho(x)^2 \ge 0,
\]
where the constant $\e>0$ is sufficiently small.
Then, any local Leray solution $v$ of the Navier-Stokes
equations \eqref{NS} with initial data $v_0$ is regular in the
region
\[
\Om = \bket{(x,t): \ x \in \R^3, \ 0< t < T(x)},
\]
and
\[
|v(x,t)| \le \frac {C_1}{\sqrt t}\quad \text{in }\Om.
\]
\end{corollary}

Of course this corollary is interesting only near those $x$ with $\rho(x;v_0)>0$.
It is a consequence of Corollary \ref{cor1.2}.

\bigskip

Our second set of applications is for solutions with initial data in the Herz spaces. These spaces contain self-similar and DSS solutions, and are of particular interest to the study of DSS solutions since they are weighted spaces with a particular choice of centre.
We now recall the definitions and basic properties of Herz spaces \cite{Herz,Miyachi,Tsutsui}. Let $A_k = \{ x \in \R^n: 2^{k-1} \le |x|< 2^k\}$.
For $n \in \NN$, $s \in \R$ and $p,q\in (0,\infty]$, the \emph{homogeneous Herz space} $\dt K^{s}_{p,q}(\R^n)$ is the space of functions $f \in L^p_{\loc}(\R^n \setminus \{0\})$ with finite norm
\[
\norm{f}_{\dt K^{s}_{p,q}} =
\left \{
\begin{alignedat}{3}
&  \bigg( \sum_{k \in \ZZ}  2 ^{ksq} \norm{f}_{L^p(A_k)}^q \bigg)^{1/q}  \qquad &\text{if } q<\I,
 \\
&  \sup _{k \in \ZZ} 2^{ks} \norm{f}_{L^p(A_k)}  \qquad &\text{if } q=\I.
\end{alignedat}
\right .
\]
The \emph{weak Herz space} $W\!\dt K^{s}_{p,q}(\R^n)$ are defined similarly, with $L^p(A_k)$-norm in the definition replaced by its weak version, $L^{p,\infty}(A_k)$-norm.

In what follows we take $q=\infty$, which is most suitable for our purpose.
In this case, $\dt K^{s}_{p,\infty}$-norm is equivalent to
\[
\norm{f}_{s,p}=
\sup_{x \not =0} \bket{ |x|^s\cdot \norm{f}_{L^p(B_{\frac{|x|}{2}}(x))} }.
\]
Also note $\dt K^{s}_{p,\I} \subset \dt B^{-s}_{p,\infty}$ if $1<p<\I$ and $0<s<n(1-1/p)$, see \cite[Theorem 1.6 (ii)]{Tsutsui} .

For $n=3$, let
\[
K_p:= \dt K^{1-3/p}_{p,\infty},\qquad p\ge 3.
\]
It is invariant under the scaling $f(x) \to \la f(\la x)$, i.e., \eqref{v0scaling}, the natural scaling of stationary \eqref{NS} and the following relation holds
\[
K_p \subset \dt B^{3/p-1}_{p,\infty}\subset BMO^{-1} \quad (3<p<\infty).
\]
The space $K_p$ contains those DSS %
in $ L^p_{\loc}(\R^3 \setminus \{0\})$, and thus $K_3$ contains all initial data considered in \cite{BT1,BT2}.

We are interested in the Herz spaces because they seem to be natural spaces for DSS solutions of \eqref{NS}.
The existence problem of mild solutions of \eqref{NS} in the Herz spaces has been studied extensively by Tsutsui \cite{Tsutsui}.
He proves short time existence for large data in subcritical weak Herz spaces $W\!\dt K^{s}_{p,\infty}(\R^3)$,
$0 \le s < 1-3/p$, and global existence for small data in the critical  weak Herz space
$W\!\dt K^{0}_{3,\infty}(\R^3)$.

\begin{theorem}\label{main-thm}
Let $\epsilon_0$ and $C_1$ be the constants from Theorem \ref{th1.1}.
Let $v$ be a local Leray solution of the Navier-Stokes equations \eqref{NS}
with initial data $v_0\in K_3 \cap E^2$. Assume further that there is
$\mu \in(0,1)$ such that
\begin{equation}\label{th1.3-1}
\sup_{ 0\not = x \in \R^3}\int_{B_{\mu\abs{x}}(x)}\abs{v_0}^3\le
\epsilon_0^3.
\end{equation}
Then there exist
 $\sigma_1=\sigma_1(\norm{v_0}_{K_3})>0$,
 $C_2=C_2(\norm{v_0}_{K_3})$ and
$\sigma_2=\sigma_2(\mu, \norm{v_0}_{K_3}) \in (0,\sigma_1]$,
such that, for any $R>0$,
\begin{equation}\label{th1.3-2}
\sup_x \frac{1}{R}\int_{B_R(x)}\abs{v}^2dx+\sup_x \frac{1}{R}\int_0^{\sigma_1
R^2}\int_{B_R(x)}\abs{\nabla v}^2 dxdt\le C_2,
\end{equation}
and
\begin{equation}\label{th1.3-4}
\abs{v(x,t)}\le \frac{C_1}{\sqrt{t}},\quad \text{for}\quad 0<t<\sigma_2\abs{x}^2.
\end{equation}
\end{theorem}

\noindent
{\it Comments for Theorem \ref{main-thm}:}

\begin{enumerate}
\item Estimate \eqref{th1.3-4} gives a regularity estimate for the solution below the paraboloid $t=\si_2|x|^2$, i.e., in the region bounded by $t=\si_2|x|^2$ and $t=0$.

\item Note that $K_3 \subset L^2_{\uloc}$, and \eqref{th1.3-2} is a property for all local Leray solutions, see Section \ref{S3}. However, We still need to assume $v_0 \in E^2$, since $K_3$ is not a subset of $E^2$ as the following example shows
\EQ{\label{K3notE2}
v_0(x) = \sum_{k=1}^\infty \zeta(x-2^k e_1)
}
where $\zeta$ is a smooth cut-off function supported in $B_1$ and $e_1=(1,0,0)$.

\end{enumerate}

\begin{corollary}\label{cor1.6}
Let $v$ be a local Leray solution of the Navier-Stokes equations \eqref{NS}
with initial data $v_0\in K_p$, $p>3$. Then, the same conclusion of
Theorem \ref{main-thm} is true, with the constants depending on $\norm{v_0}_{K_p}$.
\end{corollary}

This corollary is a direct consequence of Theorem \ref{main-thm} 
since $K_p \subset K_3$ and
\eqref{th1.3-1} follows for $\mu = C (\e_0/\norm{v_0}_{K^p})^{p/(p-3)}$ since
\[
\norm{v_0}_{L^3(B_{\mu\abs{x}}(x))} \le (C\mu |x|)^{1-3/p} \norm{v_0}_{L^p(B_{\mu\abs{x}}(x))} \le C \de^{1-3/p}\norm{v_0}_{K_p}.
\]

\begin{corollary}\label{cor1.7}
Let $\la>1$ and $v$ be a $\la$-DSS local Leray solution of the Navier-Stokes equations  \eqref{NS}
with $\la$-DSS initial data $v_0\in L^{3,\infty}(\R^3)$. Then $v_0\in K_3$,
\eqref{th1.3-1} holds for some $\mu>0$, and the same conclusion of Theorem
\ref{main-thm} is true.

Furthermore, there exists
$\lambda_*=\lambda_*(\mu)\in (1,2)$ such that if
$1<\lambda<\lambda_*$, then $v$ is regular at any
$(x,t)\in \R^3\times\R^+$, with
\[
\abs{v(x,t)}\le\frac{C}{\sqrt{t}}\quad   {\rm in} \ \R^3\times\R^+.
\]
\end{corollary}

This corollary answers our motivating problem.

\bigskip

The rest of the paper is organized as follows.
In Section \ref{S2} we recall auxiliary results, including the theorems of Caffarelli-Kohn-Nirenberg~\cite{CKN}, Kato~\cite{Kato84}, and the localization of divergence free vector fields.
In Section \ref{S3} we discuss various definitions and properties of local energy solutions
including a priori estimates for the pressure.
In Section \ref{S-Stokes} we prove the interior regularity result 
for the perturbed Stokes equation.   
Then we address the local analysis of the Navier-Stokes equations 
and the proof of theorem \ref{th1.1} in Section \ref{S4}.
In Section \ref{S5}  we consider local energy solutions with local $L^3$ data, and prove Corollaries \ref{cor1.2}-\ref{cor1.4}.
In Section \ref{S6}  we discuss solutions with data in Herz spaces, and  prove Theorem \ref{main-thm} and Corollaries \ref{cor1.6}-\ref{cor1.7}.
In Section \ref{S7}  Appendix 1, we prove properties of local Leray solutions stated in Section \ref{S3}.

We thank Professors Barker and Prange who kindly sent us their preprint \cite{BaPr} while we are finishing this paper. The preprint \cite{BaPr} contains a result similar to our Corollary \ref{cor1.2} for local energy weak solutions. 

\section{Preliminaries}
\label{S2}

We first recall the
following rescaled version of Caffarelli-Kohn-Nirenberg \cite[Proposition 1]{CKN}. It is formulated in the present form in \cite{NRS,Lin}, and is the basis for many regularity criteria, see e.g.~in \cite{GKT}.
For a suitable weak solution $(v,\pi)$, let
\[
C(r) = \frac 1{r^{2}} \int_{Q_r} |v|^3 dx\,dt, \quad
D(r)=  \frac 1{r^{2}} \int_{Q_r} |\pi|^{3/2} dx\,dt.
\]

\newcommand{\CKN}{\text{\tiny CKN}}

\begin{lemma}\label{CKN-Prop1}
There are absolute constants $\e_{\CKN}$ and $C_{\CKN}>0$ with the following property. Suppose $(v,\pi)$ is a suitable weak solution of NS with zero force in $Q_{r_1}$, $r_1>0$, with
\[
C(r_1) + D(r_1) \le \e_{\CKN},
\]
then $v  \in L^\infty(Q_{r_1/2})$ and
\EQ{\label{CKN-Prop1-eq1}
\norm{v}_{L^\infty(Q_{r_1/2})} \le \frac {C_{\CKN}} {r_1} .
}
\end{lemma}

We next recall the results due to Kato \cite{Kato84} and Giga \cite{Giga}.
\begin{lemma}\label{L3mildsol} There is $\e_2>0$ such that if $v_0\in L^3_\si(\R^3)$ with $\e=\norm{v_0}_{L^3}\le \e_2$, then there is a unique mild solution $v \in L^\infty(0,\infty; L^3(\R^3))$ of \eqref{NS} with zero force and initial data $v_0$ that satisfies
\begin{equation}\label{th2.1-1}
\norm{v}_{L^\infty L^3 \cap L^5_{t,x}(\R^4_+)} + \sup_{t>0}t^{1/2}\norm{v(t)}_{L^\infty(\R^3)}\le C \e .
\end{equation}
\end{lemma}

We will need the following localization lemma for divergence free vector fields.

\begin{lemma}[localization] \label{th:localization}
Let $1<p <\infty$ and $0<r<R$. There is a linear map $\Phi$ from $V = \bket{ v \in L^p(B_R;\R^3): \div v=0}$ into itself, and a constant $C = C(p,r/R)>0$ such that for $v \in V$ and $a = \Phi v\in V$, we have ${\rm{supp}}\,a\subset B_{\frac12(r+R) }$, $v=a$ in $B_{r}$,  and $\norm{a}_{L^p(B_R)} \le C \norm{v}_{L^p(B_R)} $.
\end{lemma}

\begin{proof}
We may assume $R=1$, since the general case follows by scaling $v(x) \to \td v(y)= v(Ry)$, $x\in B_R$ and $y\in B_1$.
Fix $\chi \in C^\infty_c(\R^3)$ with $\chi=1$ in $B_{r}$ and $\chi(x)=0$ if $|x|\ge \frac 12(r+1)$. We will take
\[
a = \chi v - b,
\]
where
the correction $b$ satisfies
\[
 \div b = \nb  \chi \cdot v, \quad \supp b \subset \bar A, \quad A := B_{\frac 12(r+1)} \setminus \bar B_{r}.
\]
It can be defined by $b = \Pi  (\nb \chi \cdot v)$, where $\Pi$ is a Bogovskii-map from $L^p_0(A)$ to $W^{1,p}_0(A)$, where
\[
L^p_0(A)=\{   f\in L^p(A),\ \textstyle\int_A f =0 \} ,
\]
such that $\div \Pi f = f$ and $\norm{\Pi f}_{W^{1,p}_0(A)} \le C \norm{ f}_{L^{p}_0(A)}$. Since $\int_A \nb  \chi \cdot v=0$, $b$ is defined and we have
$\norm{ b }_p \le C  \norm{\nb b }_p
\le  C \norm{v}_p$. Thus $\norm{a}_p \le C \norm{v}_p$.
\end{proof}

We will also recall the following lemma, which is proved by Jia-Sverak \cite[Lemma 2.1]{JiaSverak}.

\begin{lemma} \label{JSlemma}
Let $f$ be a nonnegative nondecreasing bounded function defined on $[0, 1]$ with the following property:
for some constants $0<\si<1$, $0<\theta <1$, $M>0$, $ \be > 0$, we have
\[
f(s)\le \th f(t)+ \frac M {(t - s)^{\be}}, \quad  \si < s <t <1.
\]
Then,
\[
\sup_{ s\in [0,\si]} f(s)\le C(\si, \theta, \beta)M,
\]
for some positive constant $C$ depending only on $\si, \theta, \beta$.
\end{lemma}

\section{Local energy solutions and Leray solutions}
\label{S3}

In this section we discuss the various definitions and properties of local energy solutions, or local Leray solutions, of \eqref{NS}.\snb{TT: which name do you prefer?
HM1226 I chose ``local Leray solution'' to distinguish from Leray-Hopf solution} We will also show a slightly better time integrality of the pressure.

The class of local Leray solutions
was introduced by Lemari\'e-Rieusset  in \cite{LR} to provide a local analogue of Leray's weak solutions \cite{leray}.
He constructed global in time local Leray solutions if $v_0$ belongs to $E^2$.
(Recall $L^q_{\uloc}$, $L^q_{\uloc,\rho}$, and $E^q$ are defined
in the paragraph after \eqref{Lq-uloc}.)
See Kikuchi-Seregin \cite{KiSe} for another construction which treats the pressure carefully.    Note that \cite{LR}, \cite{KiSe} and Jia-Sverak \cite{JiaSverak13,JiaSverak} contain alternative definitions of local Leray solutions.
As some key properties of the solutions are not explicitly included in the definition of  \cite{LR}, we will discuss only the relation of \emph{local energy solutions} of \cite{KiSe}, and \emph{local Leray solutions} of  \cite{JiaSverak13,JiaSverak}.

\begin{definition}[Local enegy solutions \cite{KiSe}]\label{def:localenergy}
A vector field $v\in L^2_{\loc}(\R^3\times [0,\infty))$ is a local energy solution to \eqref{NS} with divergence free initial data $v_0\in E^2$ if:
\begin{enumerate}
\item for some $\pi\in L^{3/2}_{\loc}(\R^3\times [0,\infty))$, the pair $(v,\pi)$ is a distributional solution to \eqref{NS},

\item for any $R>0$,
\begin{equation}\label{uniform-energy}
\esssup_{0\leq t<R^2}\,\sup_{x_0\in \R^3}\, \int_{B_R(x_0 )} |v(x,t)|^2\,dx
+ \sup_{x_0\in \R^3}\int_0^{R^2}\!\!\!\int_{B_R(x_0)} |\nabla v(x,t)|^2\,dx \,dt<\infty,
\end{equation}

\item for all compact subsets $K$ of $\R^3$ we have $v(t)\to v_0$ in $L^2(K)$ as $t\to 0^+$,

\item $v$ is suitable in the sense of Caffarelli-Kohn-Nirenberg, i.e., for all cylinders $Q$ compactly supported in  $ \R^3\times(0,\infty )$ and all non-negative $\phi\in C_c^\infty (Q)$, we have
\EQ{\label{CKN-LEI}
&\int |v|^2\phi(t) \,dx +2\int_0^t\!\! \int |\nabla v|^2\phi\,dx\,dt
\\&\leq %
\int_0^t\!\!\int |v|^2(\partial_t \phi + \Delta\phi )\,dx\,dt +\int_0^t\!\!\int (|v|^2+2\pi)(v\cdot \nabla\phi)\,dx\,dt.
}

\item for every $x_0\in \R^3$, there exists $c_{x_0}\in L^{3/2}(0,T)$ such that
\EQ{\label{pressure-decomposition-B1}
        \pi (x,t)-c_{x_0}(t)&=\frac 1 3 |v(x,t)|^2 +\int_{B_2(x_0)} K(x-y):v(y,t)\otimes v(y,t)\,dy
        \\&+\int_{\R^3\setminus B_2(x_0)} (K(x-y)-K(x_0-y)):v(y,t)\otimes v(y,t)\,dy
}
in $L^{3/2}(0,T; L^{3/2}(B_{3/2}(x_0)))$, where $K(x)=\rm{p.v.} \nabla^2(\frac{1}{4\pi|x|})$.

\item for any compact supported $w \in L^2(\R^3)$,
\EQ{\label{weak-continuity}
\text{the function}\quad
t \mapsto \int_{\R^3} v(x,t)\cdot w(x)\,dx \quad \text{is continuous on }[0,\infty).
}

\end{enumerate}
\end{definition}

Property 6 in Definition \ref{def:localenergy} is rather mild:
Vector fields satisfying Properties 1-5 can be redefined at a subset of time of zero measure so that Property 6 is also satisfied, similar to Leray-Hopf weak solutions.

For any domain $\Omega \subset \R^3$, we say $(v,\pi)$ is a 
\emph{suitable weak solution} in $\Omega \times (0,T)$ if 
it satisfies \eqref{NS} in the sense of distributions in 
$\Omega \times (0,T)$, 
$$
v \in L^\infty L^2(Q) \cap L^2 \dot{H}^1(Q), \quad  \pi \in L^{3/2}(Q),
$$
and local energy inequality \eqref{CKN-LEI} 
for all cylinders $Q$ compactly supported in  $ \Omega \times(0,T)$ and all non-negative $\phi\in C_c^\infty (Q)$.

\begin{definition}[local Leray solutions of \cite{JiaSverak13,JiaSverak}]\label{def:Leray}
A vector field $v\in L^2_{\loc}(\R^3\times [0,\infty))$ is a local Leray solution to \eqref{NS} with divergence free initial data $v_0\in E^2$ if
properties 1--4 of Definition \ref{def:localenergy} are satisfied, while properties 5--6 are replaced by
\begin{enumerate}
\item[7.] for any $R>0$,
\begin{equation}\label{spatial-decay}
\lim_{|x_0|\to \infty} \int_0^{R^2} \!\!\! \int_{B_R(x_0 )} | v(x,t)|^2\,dx \,dt=0,
\end{equation}
\end{enumerate}
\end{definition}

On one hand, Definition \ref{def:localenergy} requires the pressure decomposition formula
\eqref{pressure-decomposition-B1} in $B_1(x_0)$ for every $x_0$.
On the other hand, in Definition \ref{def:Leray}, the formula \eqref{pressure-decomposition-B1} is replaced by the decay condition
\eqref{spatial-decay} at spatial infinity.
Jia and  \v Sver\'ak claim in \cite{JiaSverak13,JiaSverak} that, if $v$ exhibits this decay, then the pressure  decomposition formula \eqref{pressure-decomposition-B1} is valid.  Since the decay property is easier to verify for a given solution, this justifies using it in place of the explicit pressure formula \eqref{pressure-decomposition-B1}. Since  \cite{JiaSverak13,JiaSverak} do not provide a proof and we need a better estimate for the pressure, we will prove the equivalence of the two definitions,
using ideas contained in a recent preprint of Maekawa, Prange and the second author \cite{MaMiPr}
 on the construction of local energy solutions in the half space.

The following lemma from \cite{KiSe} shows that  a local energy solution is also a local Leray solution.

\begin{lemma}[\cite{KiSe}, Lemma 2.2] \label{th3.3}
Let $\chi_R(x)=\chi(\frac xR)$ and $\chi(x)$ be a smooth cut-off function in $\R^3$ so that $\chi(x)=0$ for $|x|<1$ and $\chi(x)=1$ for $|x|>2$.
A local energy solution $(v,\pi)$ with divergence free initial data $v_0\in E^2$ in the sense of Definition \ref{def:localenergy} has the decay estimate
\EQ{
\esssup_{0<t<T} \al_R(t) + \beta_R(T) + \ga_R^{\frac23}(T) + \de_R^{\frac43}(T)
\le C(T,A) \bket{ \norm{\chi_R v_0}_{L^2_\uloc}^2 + R^{-2/3}},
}
for any $T\in (0,\infty)$ and $ R\in (1,\infty)$,
where $A = \esssup_{0<t<T} \al_0(t) + \beta_0(T) + \ga_0^{\frac23}(T) $ and
\[
\al_R(t) = \norm{ \chi_R v(\cdot,t)}_{L^2_\uloc}^2 , \quad \beta_R(t) =\sup_{x_0\in \R^3} \int_0^t \int _{B_1(x_0)}  |\chi_R \nb v|^2
\]
\[
 \ga_R(t) =\sup_{x_0\in \R^3} \int_0^t \int _{B_1(x_0)}  |\chi_R  v|^3, \quad
 \de_R(t) =\sup_{x_0\in \R^3} \int_0^t \int _{B_{3/2}(x_0)}  |\chi_R (p-c_{x_0})|^{3/2}.
\]
In particular, a local energy solution $(v,\pi)$ to \eqref{NS}  satisfies \eqref{spatial-decay} and is a local 
 Leray solution to \eqref{NS}  in the sense of Definition \ref{def:Leray}.

\end{lemma}

It is standard to see that $A$ is finite by using properties 2--5 and the Sobolev embedding. %

The following lemma shows that  a local Leray solution is also a local energy solution.  It is stated in \cite{JiaSverak13,JiaSverak} without a proof.
We will give a proof in Appendix 1 (\S\ref{S7}) for the sake of completeness.

\begin{lemma}[pressure decomposition] \label{pressure-decomposition}
Suppose  $(v,\pi)$ is a local Leray solution to \eqref{NS} with divergence free initial data $v_0\in E^2$ in the sense of Definition \ref{def:Leray}. For any $x_0\in \R^3$, $r>0$, and $T>0$, we have for $(x,t) \in Q:= B_r(x_0)\times (0,T)$,
\EQ{
\label{pressure-decomposition-Br}
\pi(x,t) &=\pi_\loc(x,t) + \pi_\far(x,t) + c_{x_0,r}(t)
\\
\pi_\loc(x,t) & = -\frac 13  |v|^2(x,t) + \int_{B_{2r}(x)} K(x-y):
( v \otimes v)(y,t)\, dy %
\\
\pi_\far(x,t) & = \int_{\R^3 \backslash B_{2r}(x)} 
\bkt{K(x-y) - K(x_0-y)}:(v \otimes v)(y,t) \, dy %
}
 for some function $c_{x_0,r}(t)\in L^{3/2}(0,T)$.
In particular, $(v,\pi)$ is a
 local energy solution to \eqref{NS} with initial data $v_0$ in the sense of Definition \ref{def:localenergy}.
\end{lemma}

The decomposition \eqref{pressure-decomposition-Br} is stronger than \eqref{pressure-decomposition-B1} since the radius $r$ is arbitrary.

We do not have a bound of $c_{x_0,r}$, which is not needed anyway, since the quantity in the equation \eqref{NS} is $\nb \pi$. Formally
$
c_{x_0,r}(t)  = 
\int_{\R^3 \backslash B_{2r}(x)}  K(x_0-y)(  v \otimes v)(y,t)\,dy
$,
but the integral does not converge.

With both Lemmas \ref{th3.3} and \ref{pressure-decomposition}, we can treat local energy solutions and local Leray solutions as the same.%
\medskip

The following lemma is the a priori bounds for the local Leray solutions.
In particular the first estimate \eqref{ineq.apriorilocal} is proved in \cite[Lemma~2.2]{JiaSverak13}.
We will give a proof in Appendix 1 (\S\ref{S7}).

\begin{lemma}[a priori bounds] \label{apriori-bounds}
Suppose  $(v,\pi)$ is a local Leray solution to \eqref{NS} with divergence free initial data $v_0\in E^2$ and $\pi$ is decomposed as in Lemma \ref{pressure-decomposition} in every $B_r(x_0)$.
For any $s,q>1$ with   $\frac 2{s}+ \frac 3{q} = 3$, 
there exists a positive constant $C(s,q)$ such that
\begin{equation}\label{ineq.apriorilocal}
\esssup_{0\leq t \leq \sigma r^2}\sup_{x_0\in \RR^3} \frac 1r\int_{B_r(x_0)}| v|^2  \,dx + \sup_{x_0\in \RR^3}   \frac 1r\int_0^{\sigma r^2}\!\!\!\int_{B_r(x_0)} |\nabla  v|^2\,dx\,dt <C_0 N_r ,
\end{equation}
\begin{equation}
\label{p.apriorilocal}
\sup_{x_0 \in \R^3} \frac 1r \norm{\pi -c_{x_0,r}(t)}_{L^s(0,\si r^2; L^q(B_r(x_0)))}   \le C(s,q) N_r
\end{equation}
where
\begin{equation*}
N_r = \sup_{x_0\in \RR^3} \frac 1r \int_{B_r(x_0)}  {|v_0|^2}\,dx,\quad
{\sigma=} \sigma(r) =c_0\, \min\big\{(N_r)^{-2} , 1  \big\},
\end{equation*}
for universal constants $C_0$ and $c_0>0$.%
\end{lemma}
Note that both estaimtes are stated as the dimension free form in the sense of \cite{CKN}.
Instead of \eqref{p.apriorilocal}, the following is given in \cite[(3.6)]{JiaSverak},
\begin{equation}\label{p.apriorilocal2}
 \sup_{x_0\in \RR^3}   \frac1{r^2}\int_0^{\sigma r^2}\int_{B_r(x_0)} |\pi-c_{x_0,r}(t)|^{3/2}\,dx\,dt <C N_r^{3/2},
\end{equation}
which is a consequence of \eqref{p.apriorilocal} by H\"older inequality.
 In \cite{KiSe} and \cite{JiaSverak13,JiaSverak}, the estimate of $\norm{\pi-c_{x_0,r}(t)}$ is only in $L^{3/2} (B_r(x_0)\times (0,T))$. This is however not sufficient for our purpose:  We need the exponent for time integration to be larger than $3/2$ for the application to Theorem \ref{th1.1}.  In fact, that $p \in L^{5/3}_{t,x,\loc}$ seems to be implicitly used in the proof of \cite[Theorem 3.1]{JiaSverak}, as explained below \cite[(3.3)]{Tsai-DSSI}. In the regularity theory for \eqref{NS}, one can often improve the spatial regularity but not the temporal regularity. Hence it is advantageous to start with a higher exponent for the time integrability.

\medskip

We end this section with summarizing other fundamental properties proved in \cite{KiSe}
for the sake of completeness.
\begin{lemma}[\cite{KiSe}, Theorem 1.4] \label{th3.6}
Suppose  $(v,\pi)$ is a local energy solution to \eqref{NS} with divergence free initial data $v_0\in E^2$. Then
$v(t) \in E^2$ for all $t$,
$v(t) \in E^3$ for a.e.~$t$,
and
\EQ{
\lim_{t \to 0_+}\norm{v(t) - v_0} _{L^2_\uloc} =0.
}
\end{lemma}

\section{Perturbed Stokes system}
\label{S-Stokes}

The following interior result for the perturbed Stokes system
is similar to \cite[Lemma 2.2]{JiaSverak}.
Instead of H\"older continuity, we claim $L^q$-integrability for any finite $q$
under a weaker assumption for the perturbed term.
Recall $Q_r = B_r \times (-r^2,0)$.

\begin{proposition}\label{linear-pns}
For any $q \in [5,\infty)$, there is $\de_0=\de_0(q)>0$ such that the
following hold. For any $M> 0$, if $G\in L^5(Q_1;\R^{3\times
3})$ with $\norm{G}_{L^5(Q_1)} \le M$,
 $a \in L^5(Q_1)$ with $\div a=0$, %
and $\norm{a}_{L^5(Q_1)}\le \de_0$, $\xi \in \R^3$, $|\xi|\le 1$, $u
\in L^\oo L^2 \cap L^2 \dot{H}^1(Q_1)$, $p\in L^{3/2}(Q_1)$,
\[
\norm{u}_{L^3(Q_1)} + \norm{p}_{L^{3/2}(Q_1)} \le M
\]
solve the $a$-perturbed Stokes equations
\EQ{\label{PStokes}
u_t-\Delta u+(a+\xi)\cdot\nabla u+u\cdot\nabla a+{\rm div}\,G+\nabla
p=0,\qquad \div u=0, } in $Q_1$, then we have\snb{$u\in L^q \cap L^\infty L^{\frac{3q}5}(Q_{1/2})$ ?}
\[
u \in L^q(Q_{1/2}), \quad \norm{u}_{ L^q(Q_{1/2})} \le C(q)M.
\]
\end{proposition}

\begin{proof}[Proof of Proposition \ref{linear-pns}]
\medskip

{\bf Step 1.} Initial bounds and localization.

\underline{Bounds}. Since this equation is linear, we may assume $M=1$. Because $a \in L^5$ and $u$ is in energy space, we can prove  local energy inequality for  the $a$-perturbed Stokes equations \eqref{PStokes}.
Fix $1-10^{-10}<\si <1$.
By the local energy inequality for \eqref{PStokes}, a calculation similar to that in \cite[page 242]{JiaSverak} shows that, for $\si<r_1<r_2<1$,
\[
E(r_1)\le
\frac{C}{(r_2-r_1)^2}+(C\norm{a }_{L^5(Q_1)}+\frac{1}{2})E (r_2),
\]
where
\[
E(r)={\rm
ess}\sup_{-r^2<t<0}\int_{B_r}\frac{\abs{u }^2}{2}dx+\int_{-r^2}^0\int_{B_r}\abs{\nabla
u }^2 dxdt.
\]
By Lemma \ref{JSlemma}, if $\norm{a }_{L^5(Q_1)}$ is sufficiently small, we have $E(\si)<C$, i.e.,
\[
\norm{u}_{L^\oo L^2 \cap L^2 \dot{H}^1(Q_\si)}\le C.
\]
Taking the divergence of \eqref{PStokes}, we get
 $p$-equation
 \[
 -\De p = \pd_i \pd_j \td G_{ij}, \quad \td G_{ij}=(a+\xi)_i u_j+u_i a_j + G_{ij}.
\]
Note that
\[
\norm{\td G}_{L^{3/2}L^{18/7}(Q_1)}
\lec (\norm{a}_{L^5} +|\xi|)\norm{u}_{L^{15/7}L^{90/17}} + \norm{G}_{L^{5}}\le C.
\]
Hence by the elliptic estimate we have %
\EQ{\label{pau.est}
\norm{p}_{L^{3/2}((-\si^2,0); L^{18/7} (B_{\si^2}))} \le \norm{\td G_{ij} }_{L^{3/2} L^{18/7} (Q_{\si})} + \norm{p}_{L^{3/2} (Q_{1})}
\le C.
}

\underline{Localization of $a$.}  By the Bogovski map,  we can solve $\tilde a: \R^3\times (-1,0)\to \R^3$ such that
\[
\div \tilde a=0, \quad \tilde a(x,t)=a(x,t) \quad \text{if } |x|<\si, \quad
a(x,t)=0 \quad \text{if } |x|>1,
\]
\[
\norm{\tilde a}_{L^5(\R^3\times (-1,0))} \le C \norm{a}_{L^5(Q_1)},
\]
for a constant $C$.

\underline{Localization of $u$.}
 Choose $\chi_0\in C^\infty_c(\R^3)$, radial, $\chi_0\ge 0$, $\chi_0=1$ on $B_{\si}$, $\chi_0=0$ on $B_1^c$. Let $\chi_k(x) = \chi_0(\si^{-k} x)$. Let $\chi = \chi_2$.
Let
\[
w=u \chi - \nb \eta, \quad
\pi = p \chi  + \pd_t \eta,
\]
where $\eta$ is the function which satisfies $\De \eta = u \cdot \nb \chi$ and it is given by
\[
\eta(x,t) = \int \frac 1{4\pi|x-y|}( u \cdot \nb \chi)(y,t)\,dy.
\]
Since $u \cdot \nb \chi$ is supported in $Q_1$,  the Calderon-Zygmund estimate shows
\[
\norm{\nb^2 \eta}_{L^{10}_t((-\si^2 ,0); (L^{30/13}_x\cap L^{6/5}_x)(\R^3))} \le C \norm{u}_{L^{10}_t (L^{30/13}_x\cap L^{6/5}_x)(Q_\si)}
\le C.
\]
By the Sobolev embedding and Riesz potential estimates, we have
\EQ{\label{Deta.est}
\norm{\nb \eta}_{L^{10}_t((-\si^2 ,0); L^{10}_x \cap L^2_x(\R^3))}
&\le C,
\\
\norm{\nb \eta}_{L^{\infty}_t((-\si^2 ,0); (L^{2}\cap L^6)(\R^3))}
& \le \norm{u}_{L^{\infty}_t(L^{6/5}\cap L^2)(Q_\si)}
\le C.
}
We also have
\EQ{\label{Deta.est2}
\norm{\nb^2 \eta}_{L^{10/3}((-\si^2 ,0)\times \R^3)} &\le  \norm{u}_{ L^{10/3}(Q_\si)} \le C
\\
\norm{\nb \eta}_{L^{10/3}_t((-\si^2 ,0); L^{\infty}(\R^3))}
&\le \norm{\nb^2 \eta}_{L^{10/3}((-\si^2 ,0)\times \R^3)} + \norm{\nb \eta}_{L^{10/3}_{t,x}}
\le C.
}

Therefore from \eqref{Deta.est}, we obtain
\EQ{
\norm{w}_{L^\oo L^2 \cap L^2 \dot{H}^1((-\si^2,0)\times \R^3)}
\le C.
}
Moreover, $w$ satisfies
\EQ{
\label{loc-PStokes}
w_t-\Delta w+(\tilde a+ \xi)\cdot\nabla w+w\cdot\nabla \tilde a+\nabla \pi=f_0+\nb F,\qquad \div w=0
\quad
\rm{in} \ \R^3 \times (-1,0),}
 \EQ{
 {\rm where} \  \qquad \qquad \qquad  f_0 & =u\Delta \chi
 +[(\tilde a+\xi)\cdot\nabla \chi] u
+[u\cdot\nabla \chi] \tilde a +p \nabla\chi +(\nb \chi)\cdot G, \quad \
\\[4pt]
F & = F_0 - \xi \otimes \nb \eta,
\\[4pt]
F_0 & =-2\nabla \chi \otimes u  -\nabla\eta\otimes \tilde a -\tilde
a \otimes \nabla \eta -\chi G.
}
Note that both $f_0$ and $F_0$ are localized. In terms of size,
\[
f_0 \approx  p + \tilde au + u+G, \quad F \approx u + \tilde a \nb
\eta + \nb \eta+ G.
\]
Recall that $ p$ and $ \tilde au$ are estimated by \eqref{pau.est}, and $\nb \eta$ by \eqref{Deta.est}.
Thus
\EQ{\label{f0F.est1}
\norm{f_0  &:\ L^{3/2} (-\si^2 ,0; (L^{18/7} \cap L^2)(\R^3))} \le C,
\\
\norm{F &:\ L^{10/3}(-\si^2 ,0; (L^{10/3} \cap L^2(\R^3))} \le C.
}

Note that $p$ appears in $f_0$, but not in $F_0$. This is very helpful because we cannot improve its time integrability.
Also note that the correction term $\nb \eta$ in the definition of $w$ is defined by the nonlocal Newtonian potential, which enables us to hide its time derivative  $\pd_t \eta$ in $\pi$, so that we don't need to estimate $\pd_t \eta$. This technique has been used, for example, in \cite{KMT,LT}.

\medskip

{\bf Step 2.} A bootstrap lemma.

In this step we prove a bootstrap lemma to improve integrability.
We will use the following potential estimates.

\begin{lemma}
\label{lem.potential-est}
Let $Q= \R^3 \times I$, $I=(0,T)$, $0<T< \infty$. Let
\[
\Phi_0 f_0(t) = \int_0^t e^{(t-s)\De}P f_0(s)ds,
\quad
\Phi_1 F(t) = \int_0^t e^{(t-s)\De}P \nb\cdot F(s)ds,
\]
for $f_0 \in L^{3/2}(I; L^r(\R^3;\R^3))$ and $F \in L^{m}(\R^3 \times I; \R^{3\times3})$.
Here $P$ denotes the Helmholtz projection on $\R^3$.
We  have
\EQ{\label{Phi0f0.est}
\norm{\Phi_0 f_0}_{L^q(Q)} \lec T^{\frac{5}{2}(\frac{1}{q}-\frac{3}{5r}+\frac{2}{15})}\norm{f_0}_{L^{3/2}(I; L^r)}, \quad
\frac1q \ge \frac3{5r} - \frac2{15},
\quad 1 < r  \le q <\infty;
}
\EQ{\label{nablaPhi0f0.est}
\norm{\nabla\Phi_0 f_0}_{L^q(Q)} \lec T^{\frac{5}{2}(\frac{1}{q}-\frac{3}{5r}-\frac{1}{15})}\norm{f_0}_{L^{3/2}(I; L^r)}, \quad
\frac1q \ge \frac3{5r} + \frac1{15},
\quad 1 < r  \le q <\infty;
}
\EQ{\label{Phi1F.est}
\norm{\Phi_1 F}_{L^q(Q)} \lec T^{\frac{5}{2}(\frac{1}{q}-\frac{1}{m}+\frac{1}{5})}\norm{F}_{L^{m}(Q)}, \quad
\frac1q \ge \frac1{{m}} - \frac1{5},
\quad 1<m \le  q < \infty.
}
\end{lemma}
\begin{proof}
By the decay estimates of $e^{(t-s)\De}P \nb$,we have
\[
\norm{\Phi_1 F(t)}_{L^q_x} \lec \int_0^t |t-s|^{-\alpha }\norm{F(s)}_{L^{m}_x}ds,
\]
where $\alpha  = \frac 32(\frac 1m-\frac 1q)+\frac12 \in [\frac12,1)$, hence $0 \le \frac 1m-\frac 1q< \frac 13$.
By the Hardy-Littlewood-Sobolev inequality and $T<\infty$, we get $\norm{\Phi_1 F}_{L^q(Q)} \lec {T^{\frac{5}{2}(\frac{1}{q}-\frac{1}{m}+\frac{1}{5})}}\norm{F}_{L^{m}(Q)}$ if
\[
\frac 1q+1 \ge \alpha  + \frac 1m, \quad i.e.,\quad
\frac 1m-\frac 1q\le \frac 15.
\]
This shows \eqref{Phi1F.est}. Similarly,
\[
\norm{\Phi_0 f_0(t)}_{L^q_x} \lec \int_0^t |t-s|^{-\alpha }\norm{f_0(s)}_{L^{r}_x}ds,
\]
where $\alpha  = \frac 32(\frac 1r-\frac 1q) \in [0,1)$, hence $0 \le \frac 1r-\frac 1q< \frac 23$.
By the Hardy-Littlewood-Sobolev inequality and $T<\infty$, we get $\norm{\Phi_0 f_0}_{L^q(Q)} \lec {T^{\frac{5}{2}(\frac{1}{q}-\frac{3}{5r}+\frac{2}{15})}}\norm{f_0}_{L^{3/2}(I; L^r)}$ if
\[
\frac 1q+1 \ge \alpha  + \frac 23, \quad i.e.,\quad
\frac1q \ge \frac3{5r} - \frac2{15},
\]
which implies $\frac 1q>\frac 1r- \frac 23$.
This shows \eqref{Phi0f0.est}.
Since the estimate \eqref{nablaPhi0f0.est} is similar to estimates above, we skip its details.
\end{proof}

Next we show the bootstrap lemma.

\begin{lemma}
\label{lem.w}
Let $ 3\le r<\infty$, $2 \le r_0 \le 9/2$
 and $2 \le r_1 \le 5$. There are small $\e= \e(r,r_0,r_1)>0$ and $\tau_0=\tau_0(r,r_0,r_1)\in (0,1)$ such that the following hold.
Let $w\in L^\infty L^2 \cap L^2 H^1(Q)$ be a weak solution of  the perturbed Stokes system \eqref{loc-PStokes} in $Q=\R^3 \times I$, $I=(t_0,t_1)$, $t_0<t_1\le t_0 + \tau_0$. Assume that $w(t_0)\in L^r \cap L^2(\R^3)$,
$f_0  \in L^{3/2}(I;L^{r_0}\cap L^2(\R^3))$, $F \in L^{r_1}(I;L^{r_1}\cap L^2)$,
$|\xi|\le 1$,
$\div \td a=0$, and $\norm{\td a }_{L^5(Q)}\le \e$.
Let
\[
N:= \norm{w(t_0)}_{L^r \cap L^2} + \norm{f_0}_{L^{3/2}(I;L^{r_0}\cap L^2)}
+ \norm{F }_{L^{r_1}(I;L^{r_1}\cap L^2)} .
\]
Then $w\in X$ with
\[
 \norm{w}_{X} \le C(r,r_0,r_1)N, \quad X=L^{\infty}L^2\cap L^2\dot H^1\cap L^q_{t,x}(Q).
\]
Here $q=\min (\frac 53r, q_0,q_1)$, $q_0 =  (\frac3{5r_0} - \frac2{15})^{-1}$, $q_1 =
( \frac1{{r_1}} - \frac1{5})^{-1}$, $\max(r_0, r_1)\le q<\infty$.
\end{lemma}

\begin{proof} Without loss of generality, we may assume $t_0=0$. We may also assume $N=1$ since it is a linear equation.
We define a sequence of approximation solutions of \eqref{loc-PStokes} by
\[
w_{1} (t) = e^{t\De} w(0) +\Phi_1F(t)+\Phi_0f_0(t),
\]
\[
w_{k+1} (t) = w_1(t)- \Phi_1 (( \tilde
a + \xi) \otimes w_k +w_k \otimes  \tilde a) (t)\quad
k=1,2,\cdots
\]

Due to the energy inequality for the Stokes system and \eqref{f0F.est1}, we have
\[
\norm{e^{t\Delta} w(0)+\Phi_1F}_{L^\infty_tL^2_x \cap L^2_t\dot{H}^1_x}
\le
C\norm{w(0)}_{L^2}+\norm{F}_{L^2_{t,x}} \le C
\]
We observe by \eqref{nablaPhi0f0.est} of Lemma \ref{lem.potential-est} that
\[
\norm{\Phi_0 f_0}_{L^\infty_tL^2_x \cap L^2_t\dot{H}^1_x}
\le Ct_1^{\frac{1}{3}} \norm{f_0}_{L^{\frac{3}{2}}_tL^2_{x}}.
\]
Therefore,
\[
\norm{w_1}_{L^\infty_tL^2_x\cap L^2_t \dot{H}^1_x}
\le
C\norm{w(0)}_{L^2}
+C\norm{F}_{L^2_{x,t}}
+Ct_1^{\frac{1}{3}} \norm{f_0}_{L^{\frac{3}{2}}_tL^2_{x}}
\le
C.
\]
On the other hand, using the $L^q_{t,x}$ estimate of
$e^{t\Delta}w(0)$ (see Giga \cite{Giga}) and
Lemma \ref{lem.potential-est}, we also have
$\norm{w_1}_{L^q_{t,x}}\le C$, and hence
$ \|w_1\|_X \le K$ for some $K>0$.

We show that $w_{k}\in X$  with $\norm{w_{k}}_X \le 2K$
provided that $t_1$ is sufficiently small by induction.
Suppose that $w_k\in X$ with $\norm{w_k}_X \le 2K$
for some $k =0,1,2,\ldots$.
The energy inequality shows that
\begin{align*}
\|\Phi_1(  ( \tilde
a + \xi) \otimes w_k +w_k \otimes  \tilde a ) \|_{L^{\infty}L^2 \cap L^2\dot{H}^1}
&\lec
 \| (\tilde
a + \xi) \otimes w_k +w_k \otimes  \tilde a  \|_{L^2_{t,x}}
\\
&\lec
 \| \tilde
a \|_{L^5_{t,x}} \|w_k\|_{L^{10/3}} +
t_1^{\frac12}\|w_k \|_{L^\infty_t L^2_x}
\end{align*}
Next using Lemma  \ref{lem.potential-est}, we can see that
\[
\norm{\Phi_1(  ( \tilde
a + \xi) \otimes w_k +w_k \otimes  \tilde a )}_{L^q}
\le
C(\norm{\tilde a }_{L^5}+t_1^{\frac{1}{2}})\norm{w_k}_{L^q}.
\]
Thus if $t_0$ and $\|\tilde a\|_{L^5_{t,x}}$ are sufficiently small,
we have
$$
\norm{\Phi_1(  ( \tilde
a + \xi) \otimes w_k +w_k \otimes  \tilde a )}_{X}
\le \frac 12 K,
$$
which shows the uniform bound of $w_k$ in $X$ by $2K$.

Denoting $\delta_k w:=w_{k+1}-w_k$, we get
\[
\delta_k w  =- \Phi_1 \bke{ ( \tilde a +
\xi) \otimes \delta_{k-1} w+\delta_{k-1} w \otimes  \tilde a},\quad k=1,2,\cdots.
\]
Following the estimates above, we see that
\[
\norm{\delta_k w (t)}_{X}\le C\bke{\norm{\tilde
a}_{L^5_{t,x}}+t_1^{1/2}}\norm{\delta_{k-1}
w(s)}_{X}.
\]
Therefore, if $\norm{\tilde
a}_{L^5_{t,x}}$ is sufficiently small, and if $t_1>0$ is so small such that
\EQ{\label{tau-small}
C t_1^{1/2} \le 1/4,
}
then $w_k$ is a Cauchy
sequence
in $X$, and converges to a mild solution $\td w$ of \eqref{PStokes} in $X$ with $\td w(0)=w(0) \in L^2 \cap L^r(\R^3)$.
By the uniqueness of the weak solution of the perturbed Stokes system \eqref{loc-PStokes} in the energy class (which can be proved by energy estimate and Gronwall inequality, using $\tilde a \in L^5_{t,x}$),
we see that $w=\tilde{w}$. This completes the proof.
\end{proof}

\medskip

{\bf Step 3.} Intermediate bounds.

Let $r=6$, $r_0=18/7$, and $r_1=10/3$.
We have $f \in L^{3/2}(-7/8,0;L^{r_0}\cap L^2)$ and $F \in L^{r_1}(-7/8,0;L^{r_1}\cap L^2)$ by \eqref{f0F.est1}.
Choose $\tau=\min(\frac 1{8}, \tau_0)$, where $\tau_0=\tau_0(r=6,r_0=18/7,r_1=10/3)$
is decided by Lemma \ref{lem.w}. Note $\frac 53r = q_0 = q_1 = 10$. Thus we take $q=10$.

Let $a_n = - n\tau /2$, and $I_n = [a_n,a_{n-1}]$, $n\in \ZZ$.
Choose smallest integer $N$ so that $I_N \subset (-\frac 78,-\frac34)$.
Since $w$ is in the energy class, we have
\[
\int_{I_{n}} \bke{ \int_{\R^3} |\nb w|^2 +|w|^2 dx}\,dt \le C, \quad 1\le n \le N.
\]
Thus, there is $t_{n} \in I_{n}$ such that
\[
{\int_{\R^3} |\nb w(t_n)|^2 + |w(t_n)|^2 \,dx \le  \frac C{\tau/2}=C}, \quad 1\le n \le N.
\]
By Sobolev imbedding, $\int |w(t_n)|^6 \,dx \le C$.
By Lemma \ref{lem.w}, we have
\[
w \in L^{10} (J_n \times \R^3), \quad J_n = (t_n, t_n+\tau )\cap (-1,0), \quad 1\le n \le N.
\]
Since
\[
t_n+\tau \ge a_n+\tau = a_{n-2} \ge t_{n-1}, \quad 2\le n \le N,
\]
we have
\[
(-3/4,0) \subset J:=\bigcup _{n=1}^N J_n = (t_N,0),\quad t_N \le -3/4.
\]
Hence
\EQ{
\norm{w}_{L^{10}  (J \times \R^3)} \le C.
}
We now show $w \in L^\infty(J; L^{10/3})$. Indeed, for $t \in J$,
\EQN{
\norm{w(t)}_{L^{10/3}_x}
&\lec \norm{w(t_N)}_{L^6 \cap L^2} + \norm{\Phi_0f_0(t)}_{L^{10/3}_x}
+ \norm{\Phi_1 F(t)}_{L^{10/3}_x} + \norm{\Phi_1 F_1(t)}_{L^{10/3}_x}
}
where $F_1 = (\td a+\xi)\otimes w + w \otimes \td a$, and
 $\Phi_0$ and $\Phi_1$ are redefined with initial time $t_N$.
Note
\EQN{
\norm{\Phi_1 F_1(t)}_{L^{10/3}_x}
&\lec \int_{t_N}^ t |t-s|^{-1/2} \norm{ F_1(s)}_{L^{10/3}_x} ds
\\
&\lec  \norm{ F_1}_{L^{10/3}_{t,x} (J \times \R^3)}
\\
&\lec (1+\norm{a}_{L^5_{t,x}}) \norm{w}_{L^{10} \cap L^{10/3} (J\times \R^3)}.
}
Similarly,
\EQN{
\norm{\Phi_0 f_0(t)}_{L^{10/3}_x}
&\lec  \norm{f_0}_{L^{3/2} L^{18/7} }
\\
\norm{\Phi_1 F(t)}_{L^{10/3}_x}
&\lec  \norm{ F}_{L^{10/3}_{t,x} }.
}
Thus
\[
\norm{w}_{L^{\infty}  (J;\ L^{10/3}(\R^3))} \le C.
\]

We also claim $\nb w \in L^{10/3}  (J\times \R^3)$ and $w \in L^{10/3}_t(J; L^\infty_x)$. Indeed,
\[
\norm{\nb w}_{L^{\frac{10}3}  (J\times \R^3)} \le \norm{\nb e^{(t-t_N)\De} w(t_N)}_{L^{\frac{10}3}  (J\times \R^3)}
+ \norm{\nb \Phi_0 f_0}_{L^{\frac{10}3}  (J\times \R^3)}  + \norm{\nb \Phi_1(F+F_1)}_{L^{\frac{10}3}  (J\times \R^3)} .
\]
By energy estimate,
\[
 \norm{\nb e^{(t-t_N)\De} w(t_N)}_{L^{\frac{10}3}  (J\times \R^3)}
 \le C\norm{\nb  w(t_N)}_{L^2  ( \R^3)} \le C.
\]
By \eqref{nablaPhi0f0.est} of Lemma \ref{lem.potential-est},
\EQN{
\norm{\nb \Phi_0 f_0(t)}_{L^{10/3}_x}
\lec  \norm{f_0}_{L^{3/2} L^{18/7} }.
}
By maximal regularity in $L^{10/3}_{t,x}$,
\[
\norm{\nb \Phi_1(F+F_1)}_{L^{\frac{10}3}  (J\times \R^3)} \le C\norm{F+F_1}_{L^{\frac{10}3}  (J\times \R^3)}\le C.
\]
Thus
$
\norm{\nb w}_{L^{10/3}  (J\times \R^3)} \le C$,
and by Sobolev inequality,
\[
\norm{w}_{L^{\frac{10}3}  (J;\ L^{\infty}(\R^3))} \lec  \norm{\nb w}_{L^{\frac{10}3}  (J \times \R^3)} +
\norm{w}_{L^{\frac{10}3}  (J \times \R^3)} \le C.
\]

By $w=u \chi - \nb \eta$, $\chi = 1$ on $B_{\si^3}$, and $\nb \eta$ estimates in \eqref{Deta.est}--\eqref{Deta.est2}, we have
\EQ{
\label{new-u-est}
\norm{u}_{L^{10}_{t,x} \cap L^{\infty} _t  L^{10/3}_x  \cap L^{10/3} _t L^{\infty} _x   ((-\frac34,0) \times B_{\si^3})}
+ \norm{\nb u}_{L^{10/3}_{t,x}  ((-\frac34,0) \times B_{\si^3})}
 \le C.
}
Note
\[
\norm{a u}_{L^{3/2}(-\frac34 ,0;\ L^5(B_{\si^3}))}\le
\norm{ a }_{L^{5}(-\frac34 ,0;\ L^5(B_{\si^3}))} \norm{ u}_{L^{15/7}(-\frac34 ,0;\ L^\infty(B_{\si^3}))}\le
C.
\]
From elliptic estimate \eqref{pau.est} again with spatial exponent $5$,
we get
\EQ{
\label{new-p-est}
\norm{p}_{L^{3/2}(-\frac34 ,0;\  L^5(B_{\si^4}))} \le
\norm{au}_{L^{3/2} L^5} + \norm{u}_{L^{3/2} L^5} + \norm{G}_{L^{3/2} L^5}+\norm{p}_{L^{3/2} L^{3/2}}
\le
C.
}
Above,
the last 4 norms are taken over $(-\frac34 ,0) \times B_{\si^3}$.

\medskip

{\bf Step 4.} Refined bounds.

We repeat the localization and define again $w=u \chi - \nb \eta$, but with $\chi$ replaced by $\chi_4(x)=\chi_0(\si^{-4} x)$.
We now have better estimates than those in Step 1 because of \eqref{new-u-est} and \eqref{new-p-est}. We first have
\[
\norm{\nb^2 \eta}_{L^{10}_t(-\frac34 ,0;\ L^{6/5}_x\cap L^{10}_x(\R^3))} \le C,
\]
\[
\norm{\nb \eta}_{L^{10}_t(-\frac34 ,0;\ L^{2}_x \cap L^\infty_x(\R^3))}\le C,
\]
and thus
\[
\norm{w}_{L^{10}_{t,x}  \cap L^\infty_t(L^2 \cap L^{\frac {10}3}) \cap L^{\frac{10}3}_t L^\infty ((-\frac34 ,0) \times \R^3)}
+ \norm{\nb w}_{ L^{\frac {10}3}_t(L^2 \cap L^{\frac {10}3})  ((-\frac34 ,0) \times \R^3)}
 \le C.
\]
Thanks to $\nb \eta \in L^\infty_{t,x}$,
we have
\[
\norm{F}_{L^5_{t,x}((-\frac34 ,0) \times \R^3)}\le C
\norm{|\td a|+|u|+|G|}_{L^5_{t,x}((-\frac34 ,0) \times B_{\si^4})}
+ \norm{\nb \eta}_{L^5_{t,x}((-\frac34 ,0) \times \R^3)}\ \le C.
\]
\[
\norm{f_0}_{L^{3/2}(-\frac34 ,0;\  L^5(B_{\si^4}))}\le C
\norm{|p|+|u|+|au|+|G|}_{L^{3/2}(-\frac34 ,0;\  L^5(B_{\si^4}))}\le C.
\]

For any $q \in (10,\infty)$, let $r=3q/5$, $r_0=9/2$, and $r_1=5$. Note $ q_0 = q_1 = \infty$, and $w \in L^{10/3}(-\frac34 ,0;\  L^r(\R^3))$.
Choose $\tau=\min(\frac 1{8}, \tau_0)$, where $\tau_0=\tau_0(r,r_0=9/2,r_1=5)$
is decided by Lemma \ref{lem.w}.

By the same argument as in Step 3, we may use time interval partition to get $u(t_n) \in L^2\cap L^r(\R^3)$, and Lemma \ref{lem.w} to get $w \in L^q_{t,x} ((-1/2,0) \times \R^3)$.

Since $w = u \chi_4 - \nb \eta$ and $\nb \eta \in L^\infty_t( L^2 \cap L^\infty_x)$,
we have, if $1/2 \le \si^5$,
\[
u \in L^q_{t,x}((-1/2,0) \times B_{1/2})).
\]
This finishes the proof of Proposition \ref{linear-pns}.
\end{proof}

\begin{remark*}
We can repeat the second part of Step 3 and prove, e.g., $w \in L^{\infty} L^q$ for any finite $q$ if $\norm{a}_{L^5_{t,x}}\le \de(q)$ is sufficiently small.
\end{remark*}

\section{Local analysis for the Navier-Stokes equations} %
\label{S4}

In this section we prove Theorem \ref{th1.1}. The proof is split into 3 subsections.

\subsection{Decay estimates for Navier-Stokes}

Let $(u,p)$ be a suitable weak solution of the following $a$-perturbed
Navier-Stokes equations in $Q=B_1\times (0,T)$, with $a \in L^5(Q)$, $\div a=0$,
\EQ{\label{PNS}
u_t-\Delta u+(a+u)\cdot \nabla u+u\cdot \nabla a+\nabla p=0,\qquad
\div u=0.
}
That is, $u\in L^\infty L^2(Q) \cap L^2 \dot H^1(Q)$, $p \in L^{3/2}(Q)$, the pair solves \eqref{PNS} in the distributional sense, and
 satisfies the \emph{perturbed local energy inequality}: For all non-negative $\phi\in C_c^\infty (Q)$, we have
\EQ{\label{PNS-LEI}
&\int |u|^2\phi(t) \,dx +2\int_0^t\!\! \int |\nabla u|^2\phi\,dx\,dt
\\&\leq
\int_0^t\!\! \int |u|^2(\partial_t \phi + \Delta\phi )\,dx\,dt +\int_0^t\!\!  \int \bke{(|u|^2(u+a)+2p u}\cdot \nabla\phi\,dx\,dt
\\& \quad + \int_0^t\!\! \int u_j a_i \pd_j( u_i \phi)\,dx\,dt.
}
This is equivalent to \eqref{CKN-LEI} for $v=u+a$ if $v$ is a weak solution of \eqref{NS} in $Q$ and $a$ is a strong solution of \eqref{NS}; see the argument after \eqref{bpib-small} for details.

Let $z=(x,t)$ and $Q_r(z)=B_r(x)\times (t-r^2,t)$. We denote
\EQ{\label{varphi.def}
\varphi (u, p,r,
z):=\bke{\frac{1}{r^2}\int_{Q_r(z)}\abs{u-(u)_{Q_r(z)}}^3}^{\frac{1}{3}}
+\bke{\frac{1}{r^2}\int_{Q_r(z)}\abs{p-(p)_{B_r(x)}(t)}^{3/2}}^{\frac{2}{3}}
}
where
\[
(u)_{Q_r(z)} = \frac 1{|Q_r(z)|} \int_{Q_r(z)} u,\qquad
(p)_{B_r(x)}(t) =  \frac 1{|B_r(x)|} \int_{B_r(x)} p(y,t)\,dy.
\]
Note that $\varphi$ is dimension-free in the sense of \cite{CKN}, and
its form is invariant under scaling.

\begin{lemma}[Decay estimate]\label{decay}
For any $\al \in (0,1)$, there is a small $\de_0>0$ such that the following hold.
Let $(u,p)$ be a suitable weak solution to the perturbed Navier-Stokes equations \eqref{PNS} in $Q_r(z)$, with
$a\in L^5(Q_r(z))$, $\div a=0$,
$\norm{a}_{L^5(Q_r(z))}=\de\le \de_0$. %
Denote $(u)_r = (u)_{Q_r(z)}$.
 Then, for any
$\theta\in (0, 1/3)$ there exist $\e=\e(\theta,{\al})>0$ and $C=C(\al)>0$ independent of $\theta$  %
such that if%
\snb{0902: the factor $r$ is added to make it dimension free and enable us to assume $r=1$.
Compare  \cite[(2.18)]{JiaSverak} where all terms are dimension $-1$.

Question: dependence between $\de$ and $\e$?
KK: I think that $\delta$ depends on $\alpha$ and so, in principle, 
$\e$ is independent of $\alpha$ for fixed $\alpha$. However, they are relevant each other indirectly 
via the value of $\alpha$. 
}
\begin{equation}\label{decay-1}
r|(u)_r| \le 1, \quad
\varphi(u,p,r, z)+{r}\abs{(u)_r}\de <\e,
\end{equation}
then
\begin{equation}\label{decay-5}
\theta r\abs{(u)_{\theta r}}\le 1,
\end{equation}
\begin{equation}\label{decay-10}
\varphi(u,p,\theta r, z)\le C\theta^{\alpha}\bkt{\varphi(u,p,r,
z)+{r}\abs{(u)_r}\de}.
\end{equation}
\end{lemma}

\begin{proof}
Choose $q\in (5,\infty)$ such that $\al < 1- \frac 5q$. Choose $\de_0=\de_0(q(\al))$ according to Proposition \ref{linear-pns}.
Since $\varphi$ and $r(u)_r$ are dimension-free, we may assume $r=1$. We may assume
$z=0$ and skip the $z$-dependence in $\varphi$ without loss of generality.
We first show \eqref{decay-5}. Indeed,
\EQ{\label{Btheta-est}
\theta |(u)_{\theta }|
&\le \theta  |(u-(u)_1)_{\theta }|
+\theta  |(u)_1|
\\
&\le \theta  |Q_{\th }|^{-\frac13} \norm{u-(u)_1}_{L^3{(Q_{\th }) }} + \theta
\\
&\le C_3 \theta^{-\frac{2}{3}}\varphi(1)+\theta ,
}
with $C_3= |Q_1|^{-\frac13}$.
By \eqref{decay-1}, $\varphi(1) \le \e$, hence
$
\theta |(u)_{\theta }|   <1
$
if
\EQ{\label{eps.small}
\e\le \theta^{2/3}/2C_3.
}

Next we show the decay estimate \eqref{decay-10}. Here we use a
contradiction argument, following a similar argument as given in
e.g.~\cite[Lemma 3.2]{Lin} and \cite[Lemma 2.3]{JiaSverak}. Since some modification is
required, we give the details for completeness.  Suppose that this
is not the case. Then there exist solutions $(u_i, p_i)$ of \eqref{PNS}, $a_i$ and $\e_i$ with
$\lim_{i\rightarrow \infty}\e_i=0$ such that
\[
\xi_i = (u_i)_{Q_1}, \quad
|\xi_i| \le 1,\quad \norm{a_i}_{L^5(Q_1)}\le
\de_0,\quad\div a_i=0,
\]
\[
\varphi(u_i,p_i,1)+|\xi_i|\norm{a_i}_{L^5(Q_{1})}
=\e_i,
\]
\[
\varphi(u_i,p_i,\theta)\ge
C_2\theta^{\alpha} \e_i.
\]
{Here $C_2>0$ is a large constant to be chosen later.}
Setting $v_i=(u_i-(u_i)_1)/\e_i$ and
$q_i=(p_i-(p_i)_1(t))/\e_i$, it follows that
\[
\norm{v_i}_{L^3(Q_1)}+\norm{q_i}_{L^{\frac{3}{2}}(Q_1)}+\frac{|\xi_i|}{\e_i}\norm{a_i}_{L^5(Q_1)}=1,
\]
\EQ{\label{viqi-est}
\bke{\frac{1}{\theta^2}\int_{Q_{\theta}}\abs{v_i-(v_i)_{Q_{\theta}}}^3}^{\frac{1}{3}}
+\bke{\frac{1}{\theta^2}\int_{Q_{\theta}}\abs{q_i-(q_i)_{B_{\theta}}(t)}^{\frac32}}^{\frac{2}{3}}\ge
C_2\theta^{\alpha}
}
and $(v_i, q_i)$ satisfies
\[
\partial_t v_i-\Delta v_i+ (\e_i v_i + a_i + \xi_i) \cdot \nabla v_i + \bke{ v_i  +\frac{\xi_i}{\e_i}}
\cdot\nabla a_i +\nabla q_i=0,\quad \div v_i=0.
\]
Denote
\[
E_i(r)={\rm
ess}\sup_{-r^2<t<0}\int_{B_r}\frac{\abs{v_i}^2}{2}dx+\int_{-r^2}^0\int_{B_r}\abs{\nabla
v_i}^2 dxdt.
\]
By the local energy inequality for \eqref{PNS}, the calculation in \cite[page 242]{JiaSverak} shows that, for $3/4<r_1<r_2<1$,
\[
E_i(r_1)\le
\frac{C}{(r_2-r_1)^2}+(C\norm{a_i}_{L^5(Q_1)}+\frac{1}{2})E_i(r_2),
\]
By Lemma \ref{JSlemma}, if $\norm{a_i }_{L^5(Q_1)}\le \de_0$ is sufficiently small, we have $E_i(3/4)<C$ for all $i$.

By the uniform bound  $E_i(3/4)<C$ for all $i$, there
exist $(v,q)\in (L^3\times L^{3/2}) (Q_{3/4})$, $\xi \in \R^3$ and $a,G \in L^5(Q_{3/4})$ such that (if
necessary, subsequence can be taken)
\[
v_i \,\,\longrightarrow\,\, v \quad \text{strongly in
}\,\,L^3(Q_{3/4}),\qquad \xi_i \,\,\longrightarrow\,\, \xi,
\]
\[
q_i \,\,\longrightharpoonup \,\, q \quad \text{weakly in
}\,\,L^{\frac{3}{2}}(Q_{3/4}),\qquad a_i \,\,\longrightharpoonup\,\, a
\quad \text{weakly in }\,\,L^{5}(Q_{3/4}),
\]
\[
\frac{(u_i)_1}{\e_i}\otimes a_i \,\,\longrightharpoonup\,\,G \quad
\text{weakly in }\,\,L^{5}(Q_{3/4}),
\]
as $ i \to \infty$.
Furthermore, $(v, q)$ solves the linear perturbed Stokes system in $Q_{3/4}$
\[
\partial_t v-\Delta v+\xi\cdot \nabla v+a\cdot\nabla v+v\cdot\nabla a +\div G+\nabla q=0,\quad
\div v=0.
\]
Due to Proposition \ref{linear-pns}, it follows that $v\in L^q(Q_{1/2})$, $q>5$, for the exponent $q$ chosen at the beginning of the proof. Thus, by the strong convergence of $v_i$ to $v$ in
$L^3(Q_{3/4})$, we have for sufficiently large $i$
\EQ{\label{vi-est}
\bke{\frac{1}{\theta^2}\int_{Q_{\theta}}\abs{v_i-(v_i)_{\theta}}^3
dz}^{\frac{1}{3}}\le C\theta^{1-\frac{5}{q}}.
}
On the other hand, by the pressure equation, we decompose
$q_i=q^R_i+q^H_i$ such that
\[
q^R_i=(-\Delta)^{-1}{\rm div\, div}\bke{[\e v_i\otimes v_i+v_i\otimes
a_i+a_i\otimes v_i]\chi_{B_{\frac{3}{4}}}}.
\]
Here $\chi_{B_{\frac{3}{4}}}$ is the characteristic function of $B_{\frac{3}{4}}$.
We then see that $q^R_i$ converges strongly to $q^R$ in
$L^{\frac{3}{2}}(Q_{3/4})$, where $q^R$ is
\[
q^R=(-\Delta)^{-1}{\rm div\, div}\bke{[v\otimes a+a\otimes
v]\chi_{B_{\frac{3}{4}}}}.
\]
We note that $q^R\in L^l(Q_{1/2})$, where $1/l=1/q+1/5$. Therefore,
\[
\bke{\frac{1}{\theta^2}\int_{Q_{\theta}}\abs{q^R}^{\frac{3}{2}}dz}^{\frac{2}{3}}\le
C\theta^{2-\frac{5}{l}} = C\theta^{1-\frac{5}{q}} .
\]
Thus, for large $i$, we also have
\[
\bke{\frac{1}{\theta^2}\int_{Q_{\theta}}\abs{q^R_i}^{\frac{3}{2}}dz}^{\frac{2}{3}}\le
C\theta^{1-\frac{5}{q}}.
\]
Since $q^H_i$ is harmonic (in $x$) in $Q_{3/4}$, we see that
\[
\bke{\frac{1}{\theta^2}\int_{Q_{\theta}}\abs{q^H_i-(q^H_i)_{B_{\theta}}(t)}^\frac32}^{\frac{2}{3}}\le C
\theta^{\frac{5}{3}}.
\]
Adding up the above estimates,
\EQ{\label{qi-est}
\bke{\frac{1}{\theta^2}\int_{Q_{\theta}}\abs{q_i-(q_i)_{B_{\theta}}(t)}^\frac32}^{\frac{2}{3}}\le
C\theta^{1-\frac{5}{q}}.
}
The sum of  \eqref{vi-est} and \eqref{qi-est} contradicts \eqref{viqi-est} if
 we take $C_2$ sufficiently large. This completes the proof.
\end{proof}

\subsection{Regularity criterion for perturbed Navier-Stokes}

In this subsection we prove the following regularity criterion for perturbed Navier-Stokes equations \eqref{PNS}.
It is a extension of the result \cite[Theorem 2.2]{JiaSverak} for the perturbed term $a\in L^m(Q_1)$ with $m>5$. 

\begin{lemma}[Regularity criterion]\label{criterion-PNS}
For any fixed $\beta \in (0,1)$, there exist small $\e_1(\beta),\de(\beta)>0$
with the following properties:
Let $(u,p)$ be a suitable weak solution to the perturbed Navier-Stokes equations \eqref{PNS} in $Q_{3/4}$, with
$a\in L^5(Q_{3/4})$, $\div a=0$,
$\norm{a}_{L^5(Q_{3/4})}\le \de$, and
\EQ{\label{up-small}
 \int_{Q_{3/4}} |u|^3+\abs{p}^{\frac{3}{2}}\le
\epsilon_1.
}
Then we have
\begin{equation}\label{Morrey-10}
\sup_{z_0 =(x_0,t_0) \in Q_{\frac14}} \sup_{r<\frac{1}{4}}\frac{1}{r^{2+3\beta}}\int_{Q_r(z_0)}\abs{u}^3
+ \abs{p - (p)_{B_r(x_0)}(t)}^{3/2}
dz<C(\be).
\end{equation}
\end{lemma}

\begin{remark*}
Unlike \cite[Theorem 2.2]{JiaSverak} and \cite[Proposition 1]{CKN}, estimate \eqref{Morrey-10} does
not imply  H\"older continuity, but Morrey type regularity.
\end{remark*}

\begin{proof} For fixed $\be \in (0,1)$,
choose $\alpha=(1+\beta)/2$ so that $\alpha \in (\beta,1)$, and choose $\theta\in (0, 1/3)$ so that the factor $C\th^\alpha$ in \eqref{decay-10} is bounded by $\frac 12 \th^\beta$, and $\th^{1-\beta}< \frac 12$.

In the following we omit the dependence on $z_0\in Q_{1/4}$ to simplify the notation.

Let $B(r)=r\abs{(u)_{Q_{r}}}$ and $\varphi(r)$ be defined by \eqref{varphi.def}.
It is proved in \eqref{Btheta-est} for $r=1$ that
\begin{equation}\label{estimate-B}
B(\theta r)\le C_3\theta^{-\frac{2}{3}}\varphi(r)+\theta B(r),
\end{equation}
where $C_3=|Q_1|^{-1/3}$. The proof for general $r$ is the same.
Let
\[
\Psi(r)=\varphi(r)+(2C_3)^{-1}\theta^{\frac{2}{3}+\beta} B(r),
\]
where $C_3$ is the constant in \eqref{estimate-B}.
We want to show by induction that
\begin{equation}\label{estimate-psi}
\text{condition (\ref{decay-1}) is valid, and}
\quad
\Psi(\theta r)\le \theta^{\beta}\Psi(r),
\end{equation}
for $r \in I_k = [ \frac {\th^{k+1}}4 , \frac {\th^k} 4]$, for all $k \in \NN_0=\NN \cup \{ 0\}$.
Let
\[
\Psi_k = \sup_{z_0 \in Q_{1/4},\ r \in I_k} \Psi(r; z_0)
,\quad k \in \NN_0.
\]
By \eqref{up-small},
\[
\Psi_0 \le C(\be) \e_1^{1/3} \le \e
\]
 if $\e_1=\e_1(\be)$ is sufficiently small.
In particular, the condition \eqref{decay-1} is uniformly satisfied for every $z_0=(x_0,t_0)\in Q_{1/4}$ and $r \in I_0$.

Suppose that \eqref{estimate-psi} has been proved for $r \in \cup_{j<k}I_j$ and
condition \eqref{decay-1} is satisfied for $r\in I_k$ for some $k \in \NN_0$.
 By \eqref{decay-10} of Lemma \ref{decay} and \eqref{estimate-B},
\EQN{
\Psi(\theta r)&=\varphi(\theta
r)+(2C_3)^{-1}\theta^{\frac{2}{3}+\beta}B(\theta r)
\\
&\le
\frac{\theta^{\beta}}{2}\varphi(r)+\frac{\theta^{\beta}}{2}\delta
B(r)+\frac{\theta^{\beta}}{2}\varphi(r)+(2C_3)^{-1}\theta^{\frac{5}{3}+\beta}
B(r)
\\
&=\theta^{\beta}\varphi(r)+\theta^{\beta}\bke{C_3\delta\theta^{-\frac{2}{3}-\beta}+\theta^{1-\beta}
}(2C_3)^{-1}\theta^{\frac{2}{3}+\beta}B(r),
}
which is bounded by $ \theta^{\beta}\Psi(r)$ if $\delta\le \min\{\delta_0(\alpha),
(2C_3)^{-1}\theta^{\frac{2}{3}+\beta}\}$. This shows \eqref{estimate-psi} for $r\in I_k$.

As a result, $\Psi_{k+1} \le \th^\be \Psi_k \le \cdots \le \th^{(k+1)\be} \Psi_0 \le \th^{(k+1)\be} \e$. Hence
\[
r|(u)_r|=B(r)\le 2C_3 \th^{-\frac 23-\be} \Psi_{k+1} \le 2C_3 \th^{-\frac 23-\be} \th^{\be}\e \le 1
\]
by \eqref{eps.small},
\[
r|(u)_r| \de\le  1\cdot\de \le \e/2,
\]
and
\[
\varphi(u,p,r, z_0) \le  \Psi_{k+1} \le \th^{\be}\e \le \e/2
\]
for $r \in I_{k+1}$. That is, condition (\ref{decay-1}) is valid for $r \in I_{k+1}$.

By induction, we have shown \eqref{estimate-psi} for all $r\le 1/4$ and all $z_0 \in Q_{1/4}$.
In particular, if $r \in I_k$,
\[
\Psi(r,z_0) \le \Psi_k \le \th^{k\be} \e  \le C\e r^\be,
\]
which implies %
\eqref{Morrey-10}.
\end{proof}

\subsection{Proof of Theorem \ref{th1.1}}
\label{S5.3}

We now prove Theorem \ref{th1.1}.
Choose $\alpha=1/2$, $\beta=1/4$ and choose $\theta>0$ so small
that $\theta^{\alpha-\beta}$, $\theta^{1-\beta}$ and $\theta^{\beta}$
are sufficiently small in the proof of Lemma \ref{criterion-PNS}.

By Lemma \ref{th:localization},  there is $a_0 \in L^3(\R^3)$ with
\[
a_0 = v_0 \quad \text{in}\ B_{3/4}, \quad
a_0 = 0 \quad \text{in}\ B_{1}^c,\quad
\div a_0 = 0, \quad \norm{a_0}_{L^3}\le C(3,\frac34) \norm{v_0}_{L^3} \le \e_2,
\]
where $\e_2$ is the constant in Lemma \ref{L3mildsol}. By Lemma \ref{L3mildsol}, there is a unique mild solution $a$ of \eqref{NS} with zero force and initial data $a(0)=a_0$ that
satisfies \eqref{th2.1-1}. In particular,
\EQ{\label{a.est-1}
\norm{a}_ {L^5_{t,x}(\R^4_+)} \le C \e_2.
}
 Let $\pi_a$ be its corresponding pressure. We have
$\pi_a = R_iR_j a_i a_j$, and
\EQ{\label{a.est-2}
\norm{\pi_a}_{L^{5/2}_{t,x}(\R^4_+)}  \le C\norm{a}_{ L^5_{t,x}(\R^4_+)} ^2  \le C \e_2^2.
}
By the maximal regularity for the inhomogeneous Stokes system, we have
\EQ{\label{a.est-3}
\nb a \in L^{5/2}(\R^4_+), \quad  \nb \pi_a \in L^{5/3}(\R^4_+).
}

Let $b_0=v_0-a_0$, $b=v-a$, and $\pi_b = \pi - \pi_a$. Denote $T=T_1$.
Observe that $(b,\pi_b)$ is a weak solution of the $a$-perturbed Navier-Stokes
equations \eqref{PNS}
in $Q= B_1 \times (0,T)$,
with $b(x,0)=b_0(x)$, and
$b_0(x)=0$ in $B_{3/4}$.
We claim that $(b,\pi_b)$ satisfies  the perturbed
local energy inequality \eqref{PNS-LEI}.
This is also stated in \cite{JiaSverak} without detailed explanation.
Indeed, \eqref{PNS-LEI} and \eqref{CKN-LEI} for $v=a+b$ are equivalent because they differ by an equality which is the sum of the weak form of $a$-equation with $2v \phi$ as the test function and the weak form of $b$-equation with $2a \phi$ as the test function.
This equality can be proved because $a$ is a strong solution satisfying \eqref{a.est-1}--\eqref{a.est-3}.

By the interpolation, $ \norm{v}_{L^4_tL^3_x(Q)}\le C \norm{v}_{L^{\infty}_tL^2_x\cap L^2_tH^1_x (Q)}$.
Hence the assumption \eqref{th1.1-0} shows 
\[
\norm{v}_{L^{3}(Q)} \le C \norm{v}_{L^4_tL^3_x(Q)} T^{\frac1{12}}
\le C\sqrt M T^{\frac1{12}},
\]
and
\[
\norm{\pi}_{L^{3/2}(Q)} \le \norm{\pi}_{L^2_tL^{3/2}_x(Q)} T^{\frac1{6}}
\le C M T^{\frac1{6}}.
\]
Thus, if $T\le \ve^{12} M^{-6} $ with $\ve$ sufficiently small, we get
\EQ{\label{bpib-small}
\int_{0}^T \int_{B_1} \abs{b}^3+\abs{\pi_b}^{\frac{3}{2}}\le C\e + C\e^2 \le
\epsilon_1,
}
where $\epsilon_1$ is the small constant in \eqref{up-small} of Lemma \ref{criterion-PNS}.

Extend $a$, $b$, and $\pi_b$ by zero for $t<0$ and denote
$Q^T_r:=B_r\times (T-r^2,T)$. Using that $b_0(x)=0$ in $B_{3/4}$
and $\|a\|_{L^5}$ is sufficiently small, the standard energy estimate shows that $(b,\pi_b)$ is a suitable weak solution of \eqref{PNS} in $Q^T_{3/4}$ satisfying the perturbed
local energy inequality \eqref{PNS-LEI},
and $\frac34|(b)_{Q^T_{3/4}}|\le 1$. In particular, $(b,\pi_b)$ satisfies \eqref{PNS} across $t=0$ in the sense of distributions.

We can now apply Lemma \ref{criterion-PNS} to conclude
 that
\begin{equation*}%
\sup_{z_0 =(x_0,t_0) \in Q_{\frac14}^T} \sup_{r<\frac{1}{4}}\frac{1}{r^{2+3\beta}}\int_{Q_r(z_0)}\abs{b}^3
+ \abs{\pi_b - (\pi_b)_{B_r(x_0)}(t)}^{3/2}
dz<C.
\end{equation*}
Choose largest $r_1 \le 1/4$ such that $Cr_1^{3\be} \le \frac 12 \e_{\CKN}$. Hence
 \begin{equation*}%
\sup_{z_0 =(x_0,t_0) \in Q_{\frac14}^T} \sup_{r\le r_1}\frac{1}{r^{2}}\int_{Q_r(z_0)}\abs{b}^3
+ \abs{\pi_b - (\pi_b)_{B_r(x_0)}(t)}^{3/2}
dz< \frac 12 \e_{\CKN}.
\end{equation*}
We may take $T \le 4r_1^2$ as $r_1$ is an absolute constant. For $r \ge r_1$ we have
\[
\sup_{z_0  \in B_{\frac14}\times (0,T)} \sup_{r\ge r_1}\frac{1}{r^{2}}\int_{Q_r(z_0) \cap Q}\abs{b}^3
dz< \frac 1{r_1^2} C\e < \frac 12.
\]
 Since $v=a+b$ and $a\in L^5(\R^4)$ small, we have
 \begin{equation}%
\sup_{z_0  \in B_{\frac14}\times (0,T)} \sup_{0<r<\infty}\frac{1}{r^{2}}\int_{Q_r(z_0) \cap Q}\abs{v}^3
dz< 1.
\end{equation}

Now for any $z_0=(x_0,t_0) \in B_{1/4} \times (0,T)$, take $r=\frac 12 \sqrt{t_0}$. We have $r \le r_1$ and
\[
r^2 <t < 4r^2
 \quad \text{if } (x,t) \in Q_r(z_0).
\]
For this $r$,
let
\[
\td \pi = \pi_a + \pi_b - (\pi_b)_{B_r(x_0)}(t).
\]
We have
\begin{equation*}%
\frac{1}{r^{2}}\int_{Q_r(z_0)}\abs{v}^3
+ \abs{\td \pi}^{3/2}
dz<  \e_{\CKN}
\end{equation*}
if $\norm{a}_{L^5}+ \norm{\pi_a}_{L^{5/2}} \le C\e_0$ is sufficiently small. Since $v,\td \pi$ is a suitable weak solution of \eqref{NS} in $Q_r(z_0)$,
by Lemma \ref{CKN-Prop1}, we get
\begin{equation}
|v(z_0)|\le
\norm{v}_{L^\infty(Q_{r/2}(z_0))} \le \frac {C_{\CKN}}{r/2}
=  \frac {4C_{\CKN}}{\sqrt{t_0}}.
\end{equation}
This completes the proof of Theorem \ref{th1.1}.
\hfill \qed

\medskip

\section{Local Leray solutions with local $L^3$ data}
\label{S5}

In this section, we present the proofs of Corollaries \ref{cor1.2},
 \ref{cor1.3} and \ref{cor1.4}.

\begin{proof}[Proof of Corollary \ref{cor1.2}]
Let
\begin{equation}\label{scaling}
u_0(x)=\de v_0(\de x), \quad
u(x,t)=\de v(\de x, \de^2 t), \quad
p(x,t)=\de^2 \pi(\de x, \de^2 t).
\end{equation}
Then $(u,p)$ is a local Leray solution of \eqref{NS} with initial data $u_0 \in E^2$ and $\norm{u_0}_{L^3(B_1)} \le \e_0$.
By Lemma \ref{apriori-bounds} with $(s,q)=(2,3/2)$,
\begin{equation*}
\esssup_{0\leq t \leq \sigma }\sup_{x_0\in \RR^3} \int_{B_1(x_0)}| u|^2  \,dx + \sup_{x_0\in \RR^3}  \int_0^{\sigma}\!\!\!\int_{B_1(x_0)} |\nabla  u|^2\,dx\,dt <C_0 N_1 ,
\end{equation*}
\begin{equation*}
\sup_{x_0}  \norm{p -c_{x_0,1}(t)}_{L^2(0,\si ; L^{3/2}(B_1(x_0)))}   \le C(2,3/2) N_1,
\end{equation*}
where
\begin{equation*}
N_1 = \sup_{x_0\in \RR^3}  \int_{B_1(x_0)}  {|u_0|^2}\,dx= \sup_{x_0\in \RR^3} \frac 1\de \int_{B_\de(x_0)}  {|v_0|^2}\,dx. 
\end{equation*}
and $\sigma(1) =c_0\, \min\big\{(N_1)^{-2} , 1  \big\}$.
Note that we have used $\delta <1$ in the last inequality.

We may replace $p$ by $p-c_{0,1}(t)$. Then $u_0$, $u$ and $p$
satisfy the assumptions in Theorem \ref{th1.1} with $M=(C_0+C(2,3/2))N_1$.
By Theorem \ref{th1.1},
 there exists
$T_1=\e(1+M)^{-6} \in (0,\si]$ such that
$u$ is regular in $B_{1/4}  \times (0,T_1)$ with
\[
|u(x,t)| \le \frac {C_1}{\sqrt t},\quad \text{in}\quad B_{1/4}  \times (0,T_1),
\]
and
\[\sup_{z_0  \in B_{\frac14}\times (0,T_1)} \sup_{0<r<\infty}\frac{1}{r^{2}}\int_{Q_r(z_0) \cap [B_1 \times (0,T_1)]}\abs{u}^3
dz\le 1.
\]
The condition $T_1 \le \si$ is clearly satisfied if we had chosen $\e \le c_0$.
Back to $v$, we have
\[
|v(x,t)| \le \frac {C_1}{\sqrt t},\quad \text{in}\quad B_{\de/4}  \times (0,T_1 \de^2).
\]
and
\[\sup_{z_0  \in B_{\de/4}\times (0,T_1\de^2)} \sup_{0<r<\infty}\frac{1}{r^{2}}\int_{Q_r(z_0) \cap [B_{\de/4}  \times (0,T_1 \de^2)]}\abs{v}^3
dz\le 1.
\]
This completes the proof of Corollary \ref{cor1.2}.
\end{proof}

\begin{proof}[Proof of Corollary \ref{cor1.3}]
If $\de \in (0,1]$, the Corollary is a direct
consequence of Corollary \ref{cor1.2}, since smallness of $L^3$-norm in $B_\de(x_0)$ is assumed to be uniform in $x_0$. If $ \de >1$, then \eqref{cor1.3-1} is also valid for $\de =1$ and the Corollary follows from the case $\de=1$.
\end{proof}

\begin{proof}[Proof of Corollary \ref{cor1.4}]
Fix $x_0\in\R^3$ such that $\rho(x_0)=\rho(x_0;v_0)>0$. We may assume $\rho(x_0) \le 1$. By Corollary
\ref{cor1.2}, we obtain
\[
|v(x,t)| \le \frac {C_1}{\sqrt t}, \quad \text{in }\ B_{\rho(x_0)/4}(x_0) \times (0,T(x_0)),
\]
with $T(x_0)=T_1(M)\rho^2(x_0)$ and
$M = CN_{\rho(x_0)}$. The claim then follows by taking $T_1(M) = \e(1+M)^{-6}$.
\end{proof}

\section{Solutions with data in Herz spaces}
\label{S6}

In this section we prove Theorem \ref{main-thm} and Corollary \ref{cor1.7} for initial data in the Herz space $K_3$.

\medskip

\begin{lemma}\label{th7.2}
The inclusion $K_3 \subset L^2_{\mathrm{uloc}}$ holds. Moreover there exists a positive constant $C$ such that 
\[
N_R:= \sup_{x\in\R^3}\frac{1}{R} \int_{B_R(x)} \abs{v_0(x)}^2\le
CR\norm{v_0}^2_{K_3},
\]
holds for any $R>0$.
\end{lemma}

\begin{proof}
Fix $R>0$. If $|x|>2R$, then
\EQN{
\int_{B_R(x)} |v_0|^2 &\le \bke{ \int_{B_R(x)} |v_0|^3}^{2/3} \norm{1}_{L^3(B_R(x))}  \le \bke{ \int_{B_{\frac{|x|}2}(x)} |v_0|^3}^{2/3}CR
\\
& \le C \norm{v_0}_{K_3}^2 R.
}
On the other hand, if $|x|\le 2R$, then
\[
\int_{B_R(x)} |v_0|^2 \le \int_{B_{3R}(0)} |v_0|^2 \le
\sum_{k=0}^\infty \int _{x \sim 2^{-k}3R} |v_0|^2
\le  \sum_{k=0}^\infty C \norm{v_0}_{K_3}^2 2^{-k}R  = C\norm{v_0}_{K_3}^2R.
\]
These estimate show the desired bound.
\end{proof}

Note that $K_3 \not \subset E^2$, as shown by the example \eqref{K3notE2}. 

\begin{proof}[Proof of Theorem \ref{main-thm}]

By Lemma \ref{apriori-bounds} with $(s,q)=(2,3/2)$, for any $R>0$, we have 
\begin{equation}\label{Oct27-100}
\sup_{0<t<\sigma
R^2}\sup_{x\in\R^3}\frac{1}{R}\int_{B_R(x)}\abs{v}^2+\sup_{x\in\R^3}\frac{1}{R}\int_0^{\sigma
R^2}\int_{B_R(x)}\abs{\nabla v}^2
\le CN_R
\le C\norm{v_0}^2_{K_3},
\end{equation}
\begin{equation}
\label{Feb21-100}
 \sup_{x\in \RR^3}  \frac 1{R} \bke{ \int_0^{\sigma R^2}\!\!\bke{  \int_{B_R(x)} |\pi-c_{x,R}(t)|^{3/2}}^{4/3} }^{1/2} \le C\|v_0\|_{K_3}^{2},
\end{equation}
for $\si = \si(\|v_0\|_{K_3})$ independent of $R$.

Let $\mu>0$ be the small constant in \eqref{th1.3-1}.
For $x_0\in\R^3$ with $x_0\neq 0$, let $\delta = \mu|x_0|$.  By \eqref{th1.3-1}, we have
\EQ{
\int_{B_\de(x_0)} |v_0(x)|^3 dx \le \e_0^3.
}
Here $\e_0$ is the constant from Theorem \ref{th1.1}.

By the same proof of Corollary \ref{cor1.2}, we have
\[
|v(x,t)| \le \frac {C_1}{\sqrt t},\quad \text{in}\quad B_{\de/4}  \times (0,T_1 \de^2).
\]
and
\[\sup_{z_0  \in B_{\de/4}\times (0,T_1\de^2)} \sup_{0<r<\infty}\frac{1}{r^{2}}\int_{Q_r(z_0) \cap [B_{\de/4}  \times (0,T_1 \de^2)]}\abs{v}^3
dz\le 1.
\]
We do not need the assumption $\de \le 1$ since our a priori bounds \eqref{Oct27-100} and \eqref{Feb21-100}
are valid for all $R \in (0,\infty)$.

This completes the proof of Theorem \ref{main-thm}.
\end{proof}

\begin{proof}[Proof of Corollary \ref{cor1.7}]
By \cite[Lemma 3.1]{BT1}, $v_0 \in K_3$.
We will show $v_0$ satisfies \eqref{th1.3-1} for some $\mu \in (0,1)$.
Because of the discrete self-similarity, it suffice to consider the region
$A:=\{x \in \R^3; \frac12 \le |x|<\frac32 \lambda \}$. Since $v_0$ is locally $L^3$,
there exists $L>0$ such that $\int_A|v_0|^3 <L$.
We now define $r_i$ ($i=1,2,\cdots$) iteratively as
$$
r_0=\frac12,
\qquad
r_{i+1}=\sup \bket{r>0; \int_{r_i \le |x| \le r}|v_0|^3 \le \frac{\ve}{2} }
$$
unless $r_{i+1} \ge \frac{3}{2}\lambda$.
This iteration stops at finite steps, namely, there exists $j$
such that $r_j \ge \frac32\lambda$.
We then set
$$
\mu:= \min \bket{\frac12, \frac{r_i}{\lambda}~;~i=1,2,\cdots,j }.
$$
If $1\le |x|\le \lambda$, we can find some $i=1,2,\cdots,j$ such that
$$
B_{\mu|x|}(x) \subset S_i \cup S_{i+1}
$$
where $S_i:=\bket{x \in \R^3~;~r_i \le |x|\le r_{i+1}}$. Hence
by the definition of $r_i$, we see $\int_{B_{\mu|x|}(x)} |v_0|^3 \le \ve$.
Now the first part is  a direct consequence of Theorem \ref{main-thm}.
The second part can be also shown by the arguments in \cite[Lemma~3.3]{Tsai-DSSI}.
Since its verification is similar to that in \cite[Lemma~3.3]{Tsai-DSSI},
we omit the details.
\end{proof}

\section{Appendix 1: Properties of local Leray solutions}
\label{S7}

In this appendix we prove Lemmas  \ref{pressure-decomposition} and \ref{apriori-bounds}.
We will prove Lemma \ref{pressure-decomposition} following the approach of Maekawa-Miura-Prange \cite[\S3]{MaMiPr} (our case in $\R^3$ is of course simpler), and using the estimates in Maekawa-Terasawa \cite{MaTe}.

\begin{lemma}[Linear $L^p_\uloc$ estimate in $\R^d$  \cite{MaTe} Corollary 3.1]
\label{Maekawa-Terasawa}
Let $ 1\le q \le p \le \infty$. 
For $f \in L^q_\uloc(\R^d)$ and $m=0,1$, we have 
\EQ{\label{MaTe1}
\norm{\nb^m e^{t\De}f }_{L_\uloc^p} \lec t^{-m/2} (1 + t^{-\frac d2(\frac 1q - \frac 1p)}) \norm{f}_{L_\uloc^q}.
}
\EQ{\label{MaTe2}
\norm{ e^{t\De}\mathbb P \nb \cdot F }_{L_\uloc^p} \lec t^{-1/2} (1 + t^{-\frac d2(\frac 1q - \frac 1p)}) \norm{f}_{L_\uloc^q}.
}
\end{lemma}

\begin{proof}[Proof of Lemma \ref{pressure-decomposition}]

By the definition of a local energy solution, there is $A\in (0,\infty)$ such that
\[
\esssup_{0<t<T}
\norm{v(t)}_{ L^2_\uloc} ^2 +\sup_{x\in\R^3}  \int_0^T \int_{B_1(x)} |\nb v|^2 \le A.
\]
Since $v_0 \in E^2$, there are $v_0^\e \in C^\infty_{c,\si}$, $v_0^\e \to v_0$ in $L^2_{\uloc}$ as $\e\to 0$.
Fix a radial smooth nonnegative function $\phi$ such that
\[
\phi = 1 \quad \text{in}\quad B_1, \qquad \spt \phi \subset B_{3/2}.
\]
For $\phi_\e(x) = \phi(\e x)$, let
\[
v^\e (t) = e^{t \De } v_0^\e + \int_0^t  e^{(t-s)\De}\mathbb P \nb \cdot [\phi_\e v \otimes v](s) ds,
\]
and
\[
\bar v (t) = e^{t \De } v_0 + \int_0^t  e^{(t-s)\De}\mathbb P \nb \cdot [ v \otimes v](s) ds.
\]
By Lemma \ref{Maekawa-Terasawa} and $v\otimes v \in L^\infty(0,T; L^1_\uloc)$, for $0<t<T<\infty$ and $1\le q <3/2$
\EQN{
\norm{v^\e(t) , \bar v(t)}_{L^q_\uloc}
&\lec \norm{v_0} _{L^q_\uloc} + \int_0^t  s^{-1/2} (1 + s^{-\frac 32( 1 - \frac 1q)}) \norm{v\otimes v (s)}_{ L^1_\uloc}ds
\\
&\lec  \norm{v_0} _{L^2_\uloc} + (T^{1/2} + T^{\frac 3{2q}-1})\norm{v\otimes v }_{ L^\infty(0,T; L^1_\uloc)}.
}
Thus for any $1\le q <3/2$, 
\EQ{ \label{eq3.3a}
\norm{
v^\e , \bar v}_{L^\infty(0,T; L^q_\uloc)} \le C\norm{v_0} _{L^2_\uloc}+ C(T^{1/2} + T^{\frac 3{2q}-1})A.
}
We now show that
\EQ{ \label{eq3.4}
\lim_{\e \to 0_+}\norm{v^\e - \bar v}_{L^\infty(0,T; L^q_\uloc)} =0 ,\quad \text{ for  all} \ q < 3/2.
}
Indeed,\footnote{This would be quite easy if we further assume
\EQ{\label{1212a}
\norm{(1-\phi_\e) v\otimes v : L^\infty(0,T; L^1_\uloc)} \to 0 \quad \text{as}\quad \e\to 0.
} Note that our assumption \eqref{spatial-decay} is slightly weaker than such an assumption.}
for $f^\e = (1-\phi_\e) v \otimes v $ and $0<\de <\min(t,1)$,
we decompose
\EQN{
 \int_0^t e^{-(t-s){\bf A}} \mathbb{P} \nabla \cdot f^\e d s
 = (\int_{t-\de}^t + \int_0^{t-\de}) = I_1 + I_2.
}
We have
\[
\norm{I_1}_{L^{q}_{\uloc}}  \le C \int_{|t-s|<\delta} |t-s|^{-2+\frac{3}{2q}} \norm{v}_{L^\infty L^2}^2ds
\le CA \de^{\frac{3}{2q}-1}.
\]
Rewriting
\EQN{
I_2(x,t) = \int_0^{t-\de} \int_{\R^3}\nb_xS(x-y,t-s) f^\e(y,s)ds,
}
where $S(x,y)$ is the Oseen tensor of the Stokes system in $\R^3$ and using the well-known estimate
\[
|\nb_x S(x,t)| \le C_m (|x| + \sqrt t)^{-4}, \quad m \in \NN,
\]
we have 
\EQN{
| I_2(x,t)| &\le \sum_{k \in \ZZ^3} \int_0^{t-\de} \int_{B_1(x+k)} |\nb_xS(x-y,t-s) ||f^\e(y,s)|dy\,ds
\\
&\le  \sum_{k \in \ZZ^3} C ((|k|-1)_+ + \sqrt \de)^{-4} \int_0^{t-\de} \int_{B_1(x+k)} |f^\e(y,s)|dy\,ds
\\
&\le C\de^{-2} F_\e, 
}
where $F_\e = \sup_{y \in \R^3} \int_0^T \int_{B_1(y)} |f^\e(y,s)|dy\,ds$.
By assumption \eqref{spatial-decay} with $R=\max(1,\sqrt T)$, we know that
$\lim_{\e \to 0} F_\e = 0$.
For any $\tau >0$, we can first choose $\de>0$ such that $ CA \de^{\frac{3}{2q}-1}<\frac12 \tau$, and then choose $\e$ such that $C\de^{-2}F_\e <\frac12 \tau|B_1|^{-1/q}$. Then
\[
\norm{I_1 + I_2}_{L^{q}_{\uloc}}(t)  \le \tau.
\]
Because the choices of $\de$ and $\e$ are uniform in $t$, we have shown \eqref{eq3.4}.

\bigskip

Note $v^\e$ is a weak solution of the inhomogeneous Stokes system
\EQ{\label{ve.eq}
\pd_t v^\e - \De v^\e + \nb p^\e = -G_\e, \quad \div v^\e=0,
\quad  v^\e|_{t=0} = v_0^\e,
}
with associated pressure $p^\e$,
where
\[
G_\e:= \nb \cdot (\phi_\e v\otimes v), \quad
\norm{G_\e}_
{L^{\frac54}((0,T)\times \R^3)} \le C(\e,A,T).
\]
Thus, by the maximal regularity estimate,
\[
\norm{\pd_t v^\e, \nb^2 v^\e, \nb p^\e}_{L^{\frac54}((0,T)\times \R^3)}
+ \norm{ p^\e}_{L^{\frac54}(0,T; L^{\frac{15}7}(\R^3))}
\le C(\e,A,T),
\]
and
\EQ{\label{pe.formula}
p^\e(x,t) &=\int \frac {1} {4\pi |x-y|} \nb\cdot G_\e(y,t) dy
\\
&=-\frac 13 \phi_\e |v|^2(x,t) + \int K(x-y):( \phi_\e v \otimes v)(y,t) dy,
}
where $K(y)= {\rm p.v.}\nabla^2 (4\pi|y|)^{-1}$.

We now decompose $p^\e$.
For fixed $x_0\in \R^3$ and $R>0$, we
let $\psi(x) = \phi(\frac {x-x_0}{2R})$ and
decompose \eqref{pe.formula} for $x\in B_{\frac32R}(x_0)$ as
\EQN{
p^\e(x,t) &=p^\e_\loc(x,t) + p^\e_\far(x,t) + c^\e(t)
\\
p^\e_\loc(x,t) & = -\frac 13  \phi_\e |v|^2(x,t) + \int K(x-y):(\psi \phi_\e v \otimes v)(y,t) dy
\\
p^\e_\far(x,t) & = \int \bkt{K(x-y) - K(x_0-y)}:( (1-\psi)\phi_\e 
v \otimes v)(y,t) dy
\\
c^\e(t) & = \int  K(x_0-y):((1-\psi) \phi_\e v \otimes v)(y,t) dy
}
By the Calderon-Zygmund estimate, for any $q>1$ we have
\[
\int_{B_{\frac32R}(x_0)}|p^\e_\loc(x,t)|^q dx \le c_q
\int_{B_{3R}(x_0)}|v(x,t)|^{2q} dx.
\]
Thus
\[
\norm{p^\e_\loc}_{L^s(0,T; L^q(B_{\frac32R}(x_0))} \le c
\norm{v}_{L^{2s}(0,T; L^{2q}(B_{3R}(x_0))}^2 ,
\]
which is a priori bounded by $A$ if $\frac 2{2s}+ \frac 3{2q} \ge \frac 32$, i.e.,
$\frac 2{s}+ \frac 3{q} \ge 3$ and if $1<q \le 3$.

For $p^\e_\far$ we have a pointwise bound,
\EQN{
|p^\e_\far(x,t) | &\le \int_{2R<|y-x_0|} \frac {cR}{|x_0-y|^4} |v(y,t)|^2 dy
\\
& \le \sum_{0 \not = k \in \ZZ^3} \int_{B_R(x_0+Rk)} \frac {cR}{|Rk|^4} |v(y,t)|^2 dy
\\
& \le c R^{-3} \norm{v(t)}_{L^2_{\uloc,R}}^2.
}
Thus,
for $\frac 2{s}+ \frac 3{q} \ge 3$,
\EQ{ \label{pe.uest}
&\norm{p^\e_\loc}_{L^s(0,T; L^q(B_{\frac32R}(x_0))} + \norm{p^\e_\far}_{L^s(0,T; L^q(B_{\frac32R}(x_0))} \\
\le &c A +
 cT^{1/s} R^{3/q-3} \esssup_{0<t<T}  \norm{v(t)}_{L^2_{\uloc,R}}^{2}
 \\
 \le &c(T,R,s,q)A.
}

By regarding $\nb p^\e$ as a given forcing term in \eqref{ve.eq} with the uniform estimate \eqref{pe.uest}, we can invoke the local regularity estimate of the inhomogeneous heat equation, which results in, for any $\de\in (0,T)$ and $q=r^*$, i.e., $1/q=1/r-1/3$,
\begin{align}
& \norm{\pd _t v^\e,\ \nb^2 v^\e}_{ L^{s}(\de,T; L^{r}(B_R(x_0)))}
\notag
\\
&\le
 c \norm{ v^\e}_{ L^{1}(0,T; L^{1}(B_{2R}(x_0)))} + c\norm{|\phi_\e v\otimes v| +  |p^\e_\loc+ p^\e_\far|}_{ L^{s}(0,T; L^{q}(B_{\frac 32R}(x_0)))}
 \notag
   \\
&\le C(\de,T,A,R,s,q).
\label{ve.uest}
\end{align}
These uniform estimates \eqref{pe.uest} and \eqref{ve.uest}
enable us to pass limits.  %
Let $\bar p_{x_0,R} = \bar p_\loc + \bar p_\far$ in $B_{\frac32R}(x_0) \times (0,T)$ with
\EQN{
\bar p_\loc(x,t) & = -\frac 13   |v|^2(x,t) + \int K(x-y):(\psi v\otimes v )(y,t) dy,
\\
\bar p_\far(x,t) & = \int \bkt{K(x-y) - K(x_0-y)}:( (1-\psi) 
v \otimes v)(y,t) dy.
}
We have the same bound as \eqref{pe.uest},
\EQ{ \label{barp.uest}
\norm{\bar p_{x_0,R}}_{L^s(0,T; L^q(B_{\frac32R}(x_0))}  &\le c(T,R,s,q)A,
}
and 
\EQN{
p^\e_\loc + p^\e_\far &\wkto \bar p_{x_0,R} \quad \text{weakly in} \quad L^s(0,T; L^q(B_R(x_0))).
}

The limit of the weak form of \eqref{ve.eq} shows that $(\bar v,  \bar p_{x_0,R})$ is a distributional solution of
\EQ{\label{barv.eq}
\pd_t \bar v - \De \bar v + \nb \bar p = -\nb \cdot (v \otimes v), \quad \div \bar v=0,
}
and
\[
\lim _{t \to 0+} (\bar v(t), \zeta) = ( v_0, \zeta), \quad \forall \zeta \in C^\infty_c(\R^3).
\]
Note $\nb \bar p_{x_0,R} = \nb \bar p_{y_0,r}$ on $B_R(x_0) \cap B_r(y_0)$.
In particular, we may specify $\bar p$ by setting $\bar  p = \bar p_{0,1}$ on $B_1(x_0)$, and
$\bar  p = \bar p_{0,k}+c_k(t)$ on $B_k(x_0)$, $k \in \NN$, where $c_k(t)$ is the unique function of $t$ 
such that
$\bar p_{0,1} =   \bar p_{0,k}+c_k(t)$ on $B_1(x_0)$. Thus $\bar p$ is defined in $\R^3 \times (0,T)$.
Since $\nb \bar p = \nb \bar p_{x_0,R}$ on $B_R(x_0)$, we have
\[
\bar  p = \bar p_{x_0,R}+c_{x_0,R}(t) \quad \text{on} \quad B_R(x_0)
\]
for some $c_{x_0,R}(t)$. Since $\bar  p , \bar p_{x_0,R} \in L^s(0,T; L^q(B_R(x_0)))$, we get $c_{x_0,R}\in L^s(0,T)$.\snb{There is a bound. Both the fact $c_{x_0,R}\in L^s(0,T)$ and its bound do not seem useful.}
This establishes the pressure decomposition formula provided we have $\bar v=v$.
To this end,we now show the spatial decay of $\bar v$. Fix any $q<3/2$. For any $0<\de\ll 1$, by \eqref{eq3.4},%
\snb{Is this the only place we use \eqref{eq3.4}?}
there is $\e>0$ such that $\norm{\bar v - v^\e}_{L^q(0,T; L^q_\uloc)}^q\le \de$. Since $v^\e \in L^2(\R^3 \times (0,T))$, there is $R>0$ such that
\EQN{
\sup_{|x|>R} \int_0^T \int_{B_1(x)} |v^\e|^q  <\de.
}
Thus
\EQ{\label{barv.decay}
\sup_{|x|>R} \int_0^T \int_{B_1(x)} |\bar v|^q  \le C\de .
}

It remains to show $\bar v =v$. Let $u=v - \bar v$. It satisfies by \eqref{spatial-decay}%
\snb{Can we replace  \eqref{spatial-decay} by \eqref{1214b}? If not, then \eqref{spatial-decay} is a better assumption than \eqref{1214b}. See Question A.}
and \eqref{barv.decay}, for any $q<3/2$,
\EQ{
u \in L^\infty (0,T; L^q_{\uloc}), \quad
\lim_{|x|\to \infty} \int_0^T \int_{B_1(x)} |u|^q  =0,
}
and
\EQ{\label{3.16}
\pd_t u - \De u + \nb \pi=0, \quad \div u =0, \quad u(t=0)=0.
}
Let $\eta_\e(x) =\e^{-3} \phi(x/\e)$ be a mollification kernel, and
$\om _\e= \eta_\e * \curl u$. The vector field $\om_\e$ is a bounded solution of the heat equation with $\om_\e(t=0)=0$. Thus $\om_\e=0$ for all $\e>0$. For any fixed $t$, $W_{\e,t}(x) = \int_0^t \eta_\e * u(s)\, ds$ is a bounded harmonic vector field which vanishes at spatial infinity. Thus $W_{\e,t}\equiv 0$. Thus $u\equiv 0$.
\end{proof}

\begin{proof}[Proof of Lemma \ref{apriori-bounds}]
Let 
\[
A_R(t)=  \sup_{x\in \R^3} \bket{ \sup_{0<s<t} \frac 1R \int_{B_R(x)}|v(y,s)|^2 dy +
\frac{1}{R}\int^{R^2}_0 \int_{B_R(x)}|\nabla v(y,s)|^2dyds}.
\]
The standard energy estimate  as in  \cite[Lemma 3.2]{KiSe} based on
 the local energy inequality and the pressure estimate \eqref{barp.uest}  
 gives \snb{abuse of notation $A(t), A(r),A_0$. maybe stick to notation of \cite{KiSe}.

\crm{TT 0206: How about $\al_R(t)$? Note $R=1$ in  \cite{KiSe}. }
}
\[
A_R(t) \le N_R + CR^{-2}\int_0^t \bke{A_R(s) + A_R(s)^3}ds.
\]
Thus $A(t) \le 2 A_0$ for $t<\la R^2$ provided
\[
\la \le \frac{c}{1+A_0^2},
\]
which yields \eqref{ineq.apriorilocal}.
Estimate  \eqref{p.apriorilocal} then follows from \eqref{barp.uest}.
We have shown
Lemma \ref{apriori-bounds}.
\end{proof}
\section*{Acknowledgments}
We thank Professor Yohei Tsutsui for useful information about the Herz spaces.
We also thank Professors Yasunori Maekawa and Christophe Prange for valuable discussions 
 about the pressure decomposition.
The research of Kang was partially supported by NRF- 2017R1A2B4006484
and by the Yonsei University Challenge of 2017. The research of Miura was partially supported
by JSPS grant 17K05312. The research of Tsai was partially supported
by NSERC grant 261356-18.

\end{document}